\documentclass[reqno,a4paper]{article}
\usepackage{enumerate}
\usepackage[a4paper, margin=1.2in]{geometry}
\usepackage{amsfonts,amssymb,amsmath,amsthm}
\usepackage[T1]{fontenc}
\usepackage{epsfig}
\usepackage{graphics}
\usepackage[normalem]{ulem}
\usepackage{color}
\usepackage{listings}
\usepackage{version}
\usepackage{caption}
\usepackage{subcaption}
\usepackage{hyperref}
\usepackage{tikz}
\usetikzlibrary{calc}
\usepackage{tikz-3dplot}
\usepackage{pgfplots}
\usepgfplotslibrary{fillbetween}
\usetikzlibrary{patterns}
\definecolor{gG}{RGB}{ 60, 186,  84}
\definecolor{dY}{RGB}{224,165,38}
\definecolor{gB}{RGB}{48., 88.6667, 158.}
\definecolor{lightBlue}{RGB}{173, 216, 230}
\definecolor{gR}{RGB}{219,  50,  54}

\pgfplotsset{compat=1.16}

\tikzstyle{mycircle}=[circle,draw=black!80,thick, minimum size=1em]

\newtheorem{thm}{Theorem}[section]
\newtheorem{theorem}[thm]{Theorem}
\newtheorem{main theorem}[thm]{Main Theorem}

\newtheorem*{main}{Main Theorem}
\newtheorem{lemma}[thm]{Lemma}

\newtheorem{prop}[thm]{Proposition}

\theoremstyle{definition}

\newtheorem{remark}{Remark}

\makeatletter
\newcommand{\newreptheorem}[2]{\newtheorem*{rep@#1}{\rep@title}\newenvironment{rep#1}[1]{\def\rep@title{#2 \ref*{##1}}\begin{rep@#1}}{\end{rep@#1}}}
\makeatother

\newreptheorem{theorem}{Theorem}
\newreptheorem{prop}{Proposition}
\newreptheorem{lemma}{Lemma}

\newcommand{\avar}[0]{\alpha}

\newcommand{\cvar}[0]{c}
\newcommand{\Dvar}[0]{D}

\newcommand{\yvar}[0]{y}

\newcounter{relctr} 
\everydisplay\expandafter{\the\everydisplay\setcounter{relctr}{0}}

\AtBeginDocument{}

\newcommand{\PPr}{\text{Pr}}

\title{Uniqueness of the Gibbs measure for the anti-ferromagnetic Potts model on the infinite $\Delta$-regular tree for large $\Delta$}

\author{Ferenc Bencs\footnote{Korteweg de Vries Institute for Mathematics, University of Amsterdam. P.O. Box 94248 1090 GE Amsterdam  The Netherlands. Email: \texttt{ferenc.bencs@gmail.com}. Funded by the Netherlands Organisation of Scientific Research (NWO): VI.Vidi.193.068}, 
David de Boer\footnote{Korteweg de Vries Institute for Mathematics, University of Amsterdam. P.O. Box 94248 1090 GE Amsterdam  The Netherlands. Email: \texttt{daviddeboer2795@gmail.com}. Funded by the Netherlands Organisation of Scientific Research (NWO): 613.001.851}, 
Pjotr Buys\footnote{Korteweg de Vries Institute for Mathematics, University of Amsterdam. P.O. Box 94248 1090 GE Amsterdam  The Netherlands. Email: \texttt{pjotr.buys@gmail.com}. Funded by the Netherlands Organisation of Scientific Research (NWO): 613.001.851}, 
Guus Regts\footnote{Korteweg de Vries Institute for Mathematics, University of Amsterdam. P.O. Box 94248 1090 GE Amsterdam  The Netherlands. Email: \texttt{guusregts@gmail.com}. Funded by the Netherlands Organisation of Scientific Research (NWO): VI.Vidi.193.068}}

\date{\today}

\begin{document}

\maketitle 

\begin{abstract}
In this paper we prove that for any integer $q\geq 5$, the anti-ferromagnetic $q$-state Potts model on the infinite $\Delta$-regular tree has a unique Gibbs measure for all edge interaction parameters $w\in [1-q/\Delta,1)$, provided $\Delta$ is large enough. This confirms a longstanding folklore conjecture. 
\end{abstract}

\section{Introduction}

The Potts model is a statistical model, originally invented to study ferromagnetism~\cite{potts}; it also plays a central role in probability theory, combinatorics and computer science, see e.g.~\cite{Sokaltutte} for background. 

Let $G = (V,E)$ be a finite graph. The anti-ferromagnetic Potts model on the graph $G$ has two parameters, a number of \emph{states, or colors,} $q\in \mathbb{Z}_{\geq 2}$ and an edge interaction parameter $w=e^{kJ/T}$, with $J<0$ being a coupling constant, $k$ the Boltzmann constant and $T$ the temperature.
The case $q=2$ is also known as the zero-field Ising model.
A \emph{configuration} is a map $\sigma:V \to [q]:=\{1,\dots,q\}$. 
Associated with such a configuration is the \emph{weight}
$w^{m(\sigma)}$, where $m(\sigma)$ is the number of edges  $e=\{u,v\}\in E$ for which $\sigma(u)=\sigma(v)$.
There is a natural probability measure, the \emph{Gibbs measure} $\PPr_{G;q,w}[\cdot]$, on the collection of configurations $\Omega=\{\sigma:V \to [q]\}$ in which a configuration is sampled proportionally to its weight.
Formally, for a given configuration $\phi:V\to [q]$ the probability that a random configuration ${\bf \Phi}$\footnote{ We use the convention to denote random variables with capitals in boldface.} is equal to $\phi$, is given by
\begin{equation}\label{eq:gibbs finite}
\PPr_{G;q,w}[{\mathbf \Phi}=\phi] = \frac{w^{m(\phi)}}{ \sum_{ \sigma:V \rightarrow [q]} w^{m(\sigma)}},
\end{equation}
here the denominator is called \emph{partition function} of the model and we denote it by $Z(G;q,w)$ (or just $Z(G)$ if $q$ and $w$ are clear form the context).

In statistical physics the Potts model is most frequently studied on infinite lattices, such as $\mathbb{Z}^2$. 
At the cost of introducing some measure theory, the notion of a Gibbs measure can be extended to such infinite graphs, see e.g. \cite{brigt99, bright02,friedli2017}.
While at any temperature the Gibbs measure on a finite graph is unique, this is no longer the case for all infinite lattices. 
The transition from having a unique Gibbs measure to multiple Gibbs measures in terms of the temperature is referred to as a \emph{phase transition} in statistical physics~\cite{Georgii,friedli2017} and it is an important problem to determine the exact temperature, the \emph{critical temperature, $T_c$}, at which this happens. 
There exist predictions for the critical temperature on several lattices in the physics literature by Baxter~\cite{BaxterZd,Baxterq} (see also~\cite{SalasSokal} for more details and further references), but it turns out to be hard to prove these rigorously cf.~\cite{SalasSokal}. 

In the present paper we consider the anti-ferromagnetic Potts model on the infinite $\Delta$-regular tree, $\mathbb{T}_\Delta=(V,E)$, also known as the \emph{Bethe lattice}, or \emph{Cayley tree}. 
We briefly recall the formal definition of a Gibbs measure in this situation following~\cite{brigt99, bright02}. See~\cite{Roz21} for a survey on this topic in general. 

The sigma algebra is generated by sets of the form $U_{\sigma}:=\{\phi:V\to [q]\mid \phi\!\restriction_U=\sigma\}$, where $U\subset V$ is a finite set and $\sigma:U\to [q]$.
Let $w\in (0,1)$.
A probability measure $\mu$ on this sigma algebra is then called a Gibbs measure for the $q$-state anti-ferromagnetic Potts model on $\mathbb{T}_\Delta$ at $w$, if for all finite $U\subset V$ and $\mu$-a.e. $\phi:V\to [q]$ the following holds
\[
\PPr_\mu[\mathbf{\Phi}\!\restriction_{U^\circ}=\phi\!\restriction_{U^\circ}\mid \mathbf{\Phi}\!\restriction_{V\setminus U^\circ}=\phi\!\restriction_{V\setminus U^\circ}]
=\PPr_{\mathbb{T}_\Delta[U];q,w}[\mathbf{\Phi}\!\restriction_{U^\circ}=\phi|_{U^\circ}\mid \mathbf{\Phi}\!\restriction_{\partial U}=\phi|_{\partial U}],
\]
where $\partial U$ denotes the collection of vertices in $U$ that have a neighbor in $V\setminus U$ and $U^\circ:=U\setminus \partial U$. 
We note that the probability in the right-hand side of this equation is determined in the finite graph induced by $U$, $\mathbb{T}_\Delta[U]$.
Moreover, we note that for any $w\in (0,1)$ there exists at least one such Gibbs measure.

For a number of states $q\geq 2$ define 
\[
w_c:=\max\{0,1-\frac{q}{\Delta}\}.
\]
It is a longstanding folklore conjecture (cf.\cite[page 746]{blanca2020}) that the Gibbs measure is unique if and only if $w\geq w_c$ (where the inequality should be read as strict if $q=\Delta$.)
We note that using the well known Dobrushin uniqueness theorem, one obtains uniqueness of the Gibbs measure provided $w>1-\tfrac{q}{2\Delta}$ cf.~\cite{Borgsetal,SalasSokal}, which is still far way from the conjectured threshold.
The conjecture was confirmed by Jonasson for the case $w=0$~\cite{jon02}, by Srivastava, Sinclair and Thurley ~\cite{sst14} for $q=2$ (see also~\cite{Georgii}; in this case one can map the model to a ferromagnetic model since the tree is bipartite, which is much better understood), 
by Galanis, Goldberg and Yang for $q=3$~\cite{PottsLeslie3} and by three of the authors of the present paper for $q=4$ and $\Delta\geq 5$~\cite{deboer2020uniqueness}. 
Our main result is a confirmation of this conjecture for all $q\geq 5$ provided the degree of the tree is large enough.

\begin{main}
For each integer $q\geq 5$ there exists $\Delta_0 \in \mathbb{N}$ such that for each $\Delta \geq 
\Delta_0$ and each $w\in [w_c,1)$ the $q$-state anti-ferromagnetic Potts model with edge interaction parameter $w$ has a unique Gibbs measure on the infinite $\Delta$-regular tree $\mathbb{T}_{\Delta}$.
\end{main}

It has long been known that there are multiple Gibbs measures when $w<w_c$~\cite{Peruggietalgeneral,Peruggietaldiagrams}, see also~\cite{gal15}) and~\cite{Roz1,Roz2,Roz3,Roz4}.
We will briefly indicate below Lemma~\ref{lem:treerecursion} how one can deduce this.
Our main results therefore pinpoints the critical temperature for the anti-ferromagnetic Potts model on the infinite regular tree for large enough degree.
For later reference we will refer to $w_c$ as the \emph{uniqueness threshold}.

In Theorem~\ref{thm:main} below, we will reformulate our main theorem in terms of the conditional distribution of the color of the root vertex of $\mathbb{T}_{\Delta}$ conditioned on a fixed coloring of the vertices at a certain distance from the root, showing that this distribution converges to the uniform distribution as the distance tends to infinity.
We in fact show that this convergence is exponentially fast for subcritical $w$ (i.e. $w>w_c$).

\subsection{Motivation from computer science}
There is a surprising connection between phase transitions on the infinite regular tree and transitions in the computational complexity of approximately computing partition function of $2$-state models (not necessarily the Potts model) on bounded degree graphs. 
For parameters inside the uniqueness region there is an efficient algorithm for this task~\cite{weitz06,licor,sst14}, while for parameters for which there are multiple Gibbs measures on the infinite regular tree, the problem is NP-hard~\cite{Slysun,gal16}. 
It is conjectured that a similar phenomenon holds for a larger number of states.

While the picture for $q$-state models for $q\geq 3$ is far from clear, some progress has been made on this problem for the anti-ferromagnetic Potts model.
On the hardness side, Galanis, \v{S}tefankvovi\v{c} and Vigoda~\cite{gal15} showed that for even numbers $\Delta\geq 4$ and any integer $q\geq 3$, approximating the partition function of the Potts model $Z(G;q,w)$ is NP-hard on the family of graphs of maximum degree $\Delta$ for any $0\leq w< 1-q/\Delta=w_c$, which we now know to be the uniqueness threshold (for $\Delta$ large enough).
On the other side, much less is known about the existence of efficient algorithms for approximating $Z(G;q,w)$ or sampling from the measure $\PPr_{G;q,w}$ for the class of bounded degree graphs when $w>w_c$. Implicit in~\cite{bencs} there is an efficient algorithm for this problem whenever $1- \alpha q/\Delta <w\leq 1$, with $\alpha=1/e$, which has been improved to $\alpha=1/2$ in~\cite{liu2019correlation}.

For random regular graphs of large enough degree, our main result implies an efficient randomized algorithm to approximately sample from the Gibbs measure $\PPr_{G;q,w}$ for any $w_c<w\leq 1$ by a result of Blanca, Galanis, Goldberg, \v{S}tefankovi\v{c}, Vigoda and Yang~\cite[Theorem 2.7]{blanca2020}. See also~\cite{chen2023strong} for a very recent improvement.
In~\cite{eft2020}, Efthymiou proved a similar result for Erd\H{o}s-R\'enyi random graphs without the assumption that $w_c$ is equal to the uniqueness threshold on the tree.
At the very least this indicates that the uniqueness threshold on the infinite regular tree plays an important role in the study of the complexity of approximating the partition function of and sampling from the Potts model on bounded degree graphs.

\subsection{Approach}
Our approach to prove the main theorem is based on the approach from~\cite{deboer2020uniqueness} for the  cases $q=3,4$.
As is well known, to prove uniqueness it suffices to show that for a given root vertex, say $v$, the probability that $v$ receives a color $i\in [q]$, conditioned on the event that the vertices at distance $n$ from $v$ receive a fixed coloring, converges to $1/q$ as $n\to \infty$ regardless of the fixed coloring of the vertices at distance $n$.
Instead of looking at these probabilities, we look at ratios of these probabilities.
It then suffices to show that these converge to $1$.
The ratios at the root vertex $v$ can be expressed as a rational function of the ratios at the neighbors of $v$. 
See Lemma~\ref{lem:treerecursion} below.
This function is rather difficult to analyze directly and as in~\cite{deboer2020uniqueness} we analyze a simpler function coupled with a geometric approach.
A key new ingredient of our approach is to take the limit of $\Delta$, the degree of the tree, to infinity and analyze the resulting function. 
This function turns out be even simpler and behaves much better in a geometric sense.
With some work we translate the results for the limit case back to the finite case and therefore obtain results for $\Delta$ large enough.
This is inspired by a recent paper~\cite{bencs2021limit} in which this idea was used to give a precise description of the location of the zeros of the independence polynomial for bounded degree graphs of large degree.

\subsection*{Organization}
In the next section we give a more technical overview of our approach.
In particular we recall some results from~\cite{deboer2020uniqueness} that we will use and set up some terminology.
We also gather two results that will be used to prove our main theorem, leaving the proofs of these results to Section~\ref{sec:convex} and Section~\ref{sec:forward} respectively. Assuming these results, the main theorem will be proved in Subsection~\ref{ssec:proof}.

\section{Preliminaries, setup and proof outline}\label{sec:technial outline}
\subsection{Reformulation of the main result}
We will reformulate our main theorem here in terms of the conditional distribution of the color of the root vertex of $\mathbb{T}_{\Delta}$ conditioned on a fixed coloring of the vertices at a certain distance from the root.

Let $\Delta\geq 2$ be an integer.
In what follows it will be convenient to write $d=\Delta-1$. 
For a positive integer $n$ we denote by $\mathbb{T}^n_{d+1}$ the finite tree obtained from $\mathbb{T}_{d+1}$ by fixing a root vertex $r$, deleting all vertices at distance more than $n$ from the root, deleting one of the neighbors of $r$ and keeping the connected component containing $r$.
We denote the set of leaves of $\mathbb{T}^n_{d+1}$ by $\Lambda_{n}$, except when $n=0$, in which case we let $\Lambda_{0}=\{r\}$.
For a positive integer $q$ we call a map $\tau:\Lambda_{n}\to [q]$ a \emph{boundary condition at level $n$}.

The following theorem may be seen as a more precise form of our main result.
\begin{theorem}\label{thm:main}
Let $q\geq 3$ be a positive integer. There exist constants $C>0$ and $d_0>0$ such that for all integers $d\geq d_0$ and all $\alpha\in (0,1)$ the following holds for any $i\in \{1,\ldots,q\}$:
 \begin{equation}\label{equation:uniqueness critical}
        \lim_{n \to \infty} \max_{\tau: \Lambda_{n} \to [q]} \bigg\vert \PPr_{\mathbb{T}^n_{d+1},q,w_c} [ {\bf \Phi}(r) = i \ \vert\ {\bf\Phi}\!\restriction_{ \Lambda_{n}} = \tau] - \frac{1}{q}\bigg\vert = 0,
    \end{equation}
    for any boundary condition $\tau$ at level $n$ and edge interaction $w(\alpha)=1-\tfrac{\alpha q}{d+1}$,
    \begin{equation}\label{equation:uniqueness subcritical}
\bigg\vert \PPr_{\mathbb{T}^n_{d+1},q,w(\alpha)} [ {\bf \Phi}(r) = i \ \vert\ {\bf\Phi}\!\restriction_{ \Lambda_{n}} = \tau]-\frac{1}{q}\bigg\vert  \leq C\alpha^{n/2}.     
    \end{equation}
\end{theorem}
\begin{remark}\label{rem:strengthen}
We can in fact strengthen \eqref{equation:uniqueness subcritical} in two ways. 
First of all, for any $\alpha<\hat{\alpha}<1$ there exists a constant $C_{\hat{\alpha}}>0$ such that the right-hand side of~\eqref{equation:uniqueness subcritical} can be replaced by $C_{\hat{\alpha}}\hat{\alpha}^n$.
Secondly, for any fixed $d\geq d_0$ there exist a constant $C_d>0$ such that the right-hand side of~\eqref{equation:uniqueness subcritical} can be replaced by $C_{d}\alpha^n$.
\end{remark}

As is well known (see e.g.~\cite[Lemma 1.3]{deboer2020uniqueness}\footnote{The proof of that lemma in the published version of that paper contains an error; this is corrected in a more recent arXiv version: \texttt{arXiv:2011.05638 v3}.}) Theorem~\ref{thm:main} directly implies our main theorem. 
Therefore the remainder of the paper is devoted to proving Theorem~\ref{thm:main}.

We now outline how we do this.

\subsection{Log-ratios of probabilities}
Theorem~\ref{thm:main} is formulated in terms of certain conditional probabilities. 
For our purposes it turns out to be convenient to reformulate this into \emph{log-ratios} of these probabilities. 
To introduce these, we recall some relevant definitions from~\cite{deboer2020uniqueness}.
Throughout we fix an integer $q\geq 3$.

Given a (finite) graph $G=(V,E)$ and a subset $U \subseteq V$ of vertices, we call $\tau:U \to [q]$ a \emph{boundary condition} on $G$. 
We say vertices in $U$ are \emph{fixed} and vertices in $V\setminus U$ are \emph{free}.
The \emph{partition function restricted to} $\tau$ is defined as
\[
Z_{U,\tau}(G;q,w) = \sum_{\substack{ \sigma:V \rightarrow [q]\\ \sigma\restriction_{\!U} = \tau}} w^{m(\sigma)}.
\]
We just write $Z(G)$ if $U,\tau$ and $q,w$ are clear from the context. 
Given a boundary condition $\tau:U \to [q]$, a free vertex $v \in V \setminus U$ and a state $i \in [q]$ we define $\tau_{v,i}$ as the unique boundary condition on $U\cup\{v\}$ that extends $\tau$ and associates $i$ to $v$.
When $U$ and $\tau$ are clear from the context, we will denote $Z_{U \cup \{v\},\tau_{v,i}}(G)$ as $Z^{v}_{i}(G)$. Let $\tau:U\to [q]$ be a boundary condition and $v \in V$ be a free vertex. 
For any $i \in [q]$ we define the \emph{log-ratio} $\tilde{R}_{G,v,i}$ as
\[
\tilde{R}_{G,v,i}: =  \log(Z^{v}_{i}(G))-\log(Z^{v}_{q}(G)),
\]
where $\log$ denotes the natural logarithm.
Note that $\tilde{R}_{G,v,q}=0$.
We moreover remark that $\tilde{R}_{G,v,i}$ can be interpreted as the logarithm of the ratio of the probabilities that the root gets color $i$ (resp. $q$) conditioned on the event that $U$ is colored according to $\tau$.

For trees the log-ratios at the root vertex can be recursively computed from the log-ratios of its neighbors.
To describe this compactly we introduce some notation that will be used extensively throughout the paper. 
Fix $d \in \mathbb{R}_{>1}$ and let $\avar\in (0,1]$. Define the maps $G_{d,\avar;i},F_{d,\avar;i}:\mathbb{R}^{q-1} \rightarrow \mathbb{R}$ for $i \in \{1,\ldots,q-1\}$ as 
\begin{equation}\label{eq:define Gd coordinate}
G_{d,\avar;i}(x_1,\ldots,x_{q-1}) =   \frac{1-x_i}{\sum_{j=1}^{q-1}x_j + 1-  \frac{\avar \cdot q}{d+1}}
\end{equation}
and
\begin{equation}\label{eq:define Fd coordinate}
F_{d,\avar;i}(x_1,\ldots,x_q) = d \log\left(1+\frac{\avar \cdot q}{d+1} \cdot G_{d,\avar;i}(\exp(x_1),\ldots,\exp(x_{q-1}))\right).
\end{equation}
Define the map $F_{d,\avar}: \mathbb{R}^{q-1} \rightarrow \mathbb{R}^{q-1}$ whose $i$th coordinate function is given by $F_{d,\avar;i}(x_1,\ldots,x_{q-1})$ and define $G_{d,\avar}$ similarly.
To suppress notation we write $F_d=F_{d,1}$ and $G_d = G_{d,1}$. We also define $\exp(x_1, \ldots, x_{q-1}) = (\exp(x_1), \ldots, \exp(x_{q-1}))$ and $\log(x_1, \ldots, x_{q-1}) = (\log(x_1), \ldots, \log(x_{q-1}))$.
We note that $G_{d,\alpha}$ and $F_{d,\alpha}$ are analytic in $1/d$ near $0$ when viewing $d$ as a variable.
We will now use the map $F_{d,\alpha}$ to give a compact description of the tree recurrence for log-ratios.

\begin{lemma}\label{lem:treerecursion}
    Let $T = (V,E)$ be a tree, $\tau:U \to [q]$ a boundary condition on $U \subsetneq V$. Let $v$ be a free vertex of degree $d\geq 1$ with neighbors $v_1, \dots, v_d$.
   Denote $T_i$ for the tree that is the connected component of $T - v$ containing $v_i$. 
   Restrict $\tau$ to each $T_i$ in the natural way.  
   Write $\tilde{R}_{i,j}$ for the log-ratio $\tilde{R}_{T_i,v_i,j}$.
    Then for $\avar$ such that $w=1-\tfrac{\avar\cdot q}{d+1}$,
    \begin{equation}\label{eq:logtreeformula}
    (\tilde{R}_{T,v,1},\ldots, \tilde{R}_{T,v,q-1}) = \sum_{i=1}^{d} \frac{1}{d} F_{d,\avar}(\tilde{R}_{i,1},\ldots,\tilde{R}_{i,q-1}),
\end{equation}
a convex combination of the images of the map $F_{d,\alpha}$.
\end{lemma}
\begin{proof}
By focusing on the $j$th entry of the left-hand side and substituting $ R_{T,v,j}:=\exp(\tilde{R}_{T,v,j}) $, we see that \eqref{eq:logtreeformula} follows from the well known recursion for ratios 
\begin{equation}\label{eq:treeformula}
    R_{T,v,i} = \prod_{s=1}^{d} \frac{\sum_{l \in [q-1]  \setminus \{i\}}  R_{T_s,v_s,l} + w  R_{T_s,v_s,i} + 1}{\sum_{l \in [q-1] }  R_{T_s,v_s,l} + w}.
\end{equation}
See e.g.~\cite{deboer2020uniqueness} for a proof of this.
\end{proof}
We note that if the boundary condition $\tau$ is constant on the leaves of the tree $\mathbb{T}_{d+1}^n$, then the log-ratios at the root can be obtained by iterating the univariate function $f$ given by $f(x)=F_{d,\alpha}(x,\ldots,x)$ at $w=w(\alpha)$. The point $x=0$ is a fixed point of $f$; it satisfies $|f'(0)|\leq 1$ if and only if $w\geq w_c$. 
From this it is not difficult to extract that there exist multiple Gibbs measures when $w<w_c$. 

Denote $\mathbf{0}$ for the zero vector in $\mathbb{R}^{q-1}$. (Throughout we will denote vectors in boldface.)
We define for any $n \geq 1$ the set of possible log-ratio vectors 
\[
\mathcal{R}_n := \{(\tilde{R}_{\mathbb{T}_{d+1}^n,r,1}, \ldots, \tilde{R}_{\mathbb{T}_{d+1}^n,r,q-1}) \in \mathbb{R}^{q-1}| \tau:\Lambda_n \rightarrow [q] \}.
\]
Here the ratios $\tilde{R}_{\mathbb{T}_{d+1}^n,r,1}$ depend on $\tau$ but this is not visible in the notation.
The following lemma shows how the recursion from Lemma \ref{lem:treerecursion} will be used.

\begin{lemma}
    \label{lem: Tab sequence}
    Let $q\geq 3$ and $d\geq 2$ be integers. If there exists a sequence $\{\mathcal{T}_n\}_{n\geq 1}$ of convex subsets of $\mathbb{R}^{q-1}$ with the following properties:
    \begin{enumerate}
        \item \label{it:basecase}
        $\mathcal{R}_1 \subseteq \mathcal{T}_1$,
        \item \label{it:inductionstep}
        for every $n \geq 1$, $F_d(\mathcal{T}_n) \subseteq \mathcal{T}_{n+1}$,
        \item \label{it:laststep}
        for every $\epsilon > 0$ there is an $N \geq 1$ such that for all $n\geq N$, $\sup_{\bf{r}\in \mathcal{T}_n} \|\mathbf{r}\|_1 \leq \varepsilon$, 
    \end{enumerate}
then     
\begin{equation}
        \lim_{n \to \infty} \max_{\tau: \Lambda_{n,d+1} \to [q]} \bigg\vert \PPr_{\mathbb{T}^n_{d},q,w_c} [ {\bf \Phi}(r) = i \ \vert\ {\bf\Phi}\!\restriction_{ \Lambda_{n,d}} = \tau] - \frac{1}{q}\bigg\vert = 0.
    \end{equation}
\end{lemma}
\begin{proof}
The proof is straightforward and analogous to the proof of Lemma 2.3 in \cite{deboer2020uniqueness} and we therefore omit it.
\end{proof}
We note that the lemma is only stated for $\alpha=1$. An analogues statement for $\alpha\in (0,1)$ and $F_d$ replaced by $F_{d,\alpha}$ with a more accurate dependence of $N$ on $\varepsilon$ follows from a certain monotonicity of $F_{d,\alpha}$, as will be explained in the proof of Theorem~\ref{thm:main} below.

In the next section we construct a family of convex sets that allows us to form a sequence $\{\mathcal{T}_n\}_{n\geq 1}$ with the properties required by the lemma. 

\subsection{Construction of suitable convex sets}
We need the standard $q-2$-simplex, which we denote as
\[
\Delta=\left\{(t_1,\ldots,t_{q-2},1-\sum_{i=1}^{q-2}t_i)\mid t_i\geq 0 \text{ for all $i$, }\sum_{i=1}^{q-2}t_i\leq 1\right\}.
\]

The symmetric group $S_q$ acts on $\mathbb{R}^q$ by permuting entries of vectors. 
Consider $\mathbb{R}^{q-1}\subset \mathbb{R}^q$ as the subspace spanned by $\{\mathbf{e}_1 - \mathbf{e}_q, \ldots, \mathbf{e}_{q-1} - \mathbf{e}_q$\}, where $\mathbf{e}_i$ denotes the $i$th standard base vector in $\mathbb{R}^q$. 
This induces a linear action of $S_q$ on $\mathbb{R}^{q-1}$, also known as the the standard representation of $S_q$ and denoted by $\mathbf{x}\mapsto \pi\cdot \mathbf{x}$ for $\mathbf{x}\in \mathbb{R}^{q-1}$ and $\pi\in S_q$.
The following lemma shows that the map $F_{d,\avar}$ is $S_q$-equivariant for any $\avar \in (0,1]$, essentially because the action permutes the $q$ colors of the Potts model and no color plays a special role.

\begin{lemma}\label{lem:symmetry}
For any $\pi \in S_q$, any $\avar \in (0,1]$, any $\mathbf{x} \in \mathbb{R}^{q-1}$ and any $d$ we have \[\pi \cdot F_{d,\avar}(\mathbf{x}) =  F_{d,\avar}(\pi \cdot \mathbf{x}).\] 
\end{lemma}
\begin{proof}
This follows as in Section 3.1 in \cite{deboer2020uniqueness}. 
\end{proof}

Define for $\cvar\geq 0$ the half space 
\begin{equation}\label{eq:half space}
H_{\geq -\cvar}:=\left\{\mathbf{x}\in \mathbb{R}^{q-1}\mid \sum_{i=1}^{q-1} x_i\geq-\cvar\right\}.
\end{equation}
Define the set
\begin{equation}\label{eq:def P_c}
P_\cvar = \bigcap_{\pi \in S_q} \pi \cdot H_{\geq -\cvar}.
\end{equation}
Note that for each $\cvar \geq 0$ the set $P_\cvar$ equals the convex polytope 
\[
\text{conv} \big(\{(-\cvar,0,\ldots, 0), \ldots (0,\ldots,0,-\cvar),(\cvar,\ldots,\cvar)\} \big) .
\]
Denote $D_\cvar:= \text{conv} \big(\{(-\cvar,0,\ldots, 0), \ldots (0,\ldots,0,-\cvar),(0,\ldots,0)\} \big)$. 
Then we have
\begin{equation}\label{eq:fund domain}
P_\cvar = \bigcup_{\pi \in S_q}\pi \cdot D_\cvar.
\end{equation}
We refer to $D_\cvar$ as the fundamental domain of the action of $S_q$ on $\mathbb{R}^{q-1}$.

The following two propositions capture the image of $P_c$ under applications of the map $F_d$.

\begin{prop}\label{prop:convex}
Let $q\geq 3$ be an integer. Then there exists $d_1>0$ such that for all $d\geq d_1$ and $c\in [0,q+1]$, $F_d(P_c)$ is convex.
\end{prop}

\begin{prop}\label{prop:2stepforward}
Let $q\geq 3$ be an integer. There exists $d_2>0$ such that for all $d\geq d_2$ the following holds:
for any $c\in(0,q+1]$ there exists $0<c'<c$ such that
\[
F_d^{\circ 2}(P_c)\subseteq P_{c'}.
\]
\end{prop}
An intuitive explanation for why we need $F_d^{\circ 2}$ and cannot work with $F_d$ directly is that the derivative of $F_{d}$ at $\bf{0}$ is equal to $-\text{Id}$, which reflects the fact that we are dealing with an anti-ferromagnetic model, while the derivative of $F^{\circ 2}_{d}$ at $\bf{0}$ is equal to $\text{Id}$.

We postpone the proofs of the two results above to the subsequent sections. A crucial ingredient in both proofs will be to analyze the limit $\lim_{d\to \infty} F_d$.
We first utilize the two propositions to give a proof of Theorem~\ref{thm:main}.

\subsection{A proof of Theorem~\ref{thm:main}}\label{ssec:proof}
Fix an integer $q\geq 3$. Let $d_1,d_2$ be the constants from Proposition~\ref{prop:convex} and~\ref{prop:2stepforward} respectively. 
Let $d_0\geq \max\{d_1,d_2\}$ large enough to be determined below. 
Note that the log-ratios at depth $0$ are of the form $\infty \cdot \mathbf{e}_i$ and $-\infty\cdot\mathbf{1}$, where $\mathbf{1}$ denotes the all ones vector. 
This comes from the fact that the probabilities at level $0$ are either $1$ or $0$ and so the ratios are of the form $\mathbf{1}+\infty\mathbf{e}_i$ or $\mathbf{0}$. 
This implies that the log-ratios at depth $1$ are convex combinations of $F_d(\infty\cdot  \mathbf{e}_i)=d\log(1+\tfrac{-q}{d+1})\mathbf{e}_i$ and $F_d(-\infty\cdot  \mathbf{1})=d\log(1+\tfrac{q}{d+1-q})\mathbf{1}.$
So for $d\geq d_0$ and $d_0$ large enough they are certainly contained in $P_{q+1}$.

We start with the proof of~\eqref{equation:uniqueness critical}. We construct a decreasing sequence $\{c_n\}_{n\in \mathbb N}$ and let $\mathcal{T}_{2n-1}=P_{c_n}$.
For even $n>0$  we set $\mathcal{T}_n=F_d(P_{c_{n-1}})$, which is convex by Proposition~\ref{prop:convex}.
We set $c_1=q+1$ and for $n\geq 1$, given $c_n$, we can choose, by Proposition~\ref{prop:2stepforward}, $c_{n+1}<c_n$ so that $F_d^{\circ 2}(P_{c_n})\subseteq P_{c_{n+1}}$. Choose such a $c_{n+1}$ as small as possible.
We claim that the sequence $\{c_n\}_{n\in \mathbb{N}}$ converges to $0$. 
Suppose not then it must have a limit $c>0$. 
Choose $c'<c$ such that $F^{\circ 2}_d(P_c)\subseteq P_{c'}.$
Then for $n$ large enough we must have $F^{\circ 2}_{d}(P_{c_n})\subseteq P_{c/2+c'/2}$, contradicting the choice of $c_{n+1}$.

Since $\{c_n\}_{n\in \mathbb{N}}$ converges to $0$, it follows that the sequence $\mathcal{T}_n$ converges to $\{0\}$.
With Lemma~\ref{lem: Tab sequence} this implies \eqref{equation:uniqueness critical}.

To prove the second part let $\avar\in (0,1)$. Consider the decreasing sequence $\{\cvar_n\}_{n\in \mathbb N}$ with $\cvar_n=(q+1)\alpha^{n-1}$. Set $\mathcal{T}_{2n-1}=P_{\cvar_n}$ and $\mathcal{T}_{2n}=F_{d,\avar}(P_{\cvar_{n-1}})$.
We use the following observation.
\begin{lemma}\label{lem:increasingd and alpha}
For any $\avar \in (0,1]$, any $\mathbf{x} \in \mathbb{R}^{q-1}$ and any integer $d$ there is $d' \geq d$ such that $F_{d,\avar}(\mathbf{x}) = \frac{d}{d'} \cdot F_{d'}(\mathbf{x})$. Moreover, $\frac{d}{d'}\leq \alpha$.
\end{lemma}
\begin{proof}
When viewing $\avar$ and $d$ as variables, $\tfrac{1}{d}F_{d,\avar;i}$ only depends on the ratio $\tfrac{\avar}{d+1}$. 
Therefore the first statement of the lemma holds with $d'$ defined by $\tfrac{\avar}{d+1}=\tfrac{1}{d'+1}.$
Since $\tfrac{d}{d'}=\tfrac{\avar d}{d+1-\avar}$, the second statement also holds.
\end{proof}

The lemma above implies that $F_{d,\alpha}(P_{\cvar_n})=\tfrac{d}{d'}\cdot F_{d'}(P_{\cvar_n})$ and hence is convex for each $c_n.$
It moreover implies that 
\[F^{\circ 2}_{d,\alpha}(P_{\cvar_n})\subset \alpha F_{d'}(\alpha F_{d'}( P_{\cvar_n})))\subset\alpha P_{\cvar_{n}}=P_{\cvar_{n+1}}.\] 
By basic properties of the logarithm, \eqref{equation:uniqueness subcritical} now quickly follows.
This finishes the proof of Theorem~\ref{thm:main}.

The strengthening mentioned in Remark~\ref{rem:strengthen} can be derived from the fact that the derivative of $F_{d,\alpha}$ at $\mathbf{0}$ is equal to $\frac{-\alpha d}{d+1-\alpha}\text{Id}$.
Note that $\frac{\alpha d}{d+1-\alpha}<\alpha$ for all $\alpha\in (0,1)$ and $d$. 
Therefore on a small enough open ball $B$ around $\mathbf{0}$ the operator norm of the derivative of $F_{d,\alpha}$ can be bounded by $\hat{\alpha}$ for all $d\geq d_0$ (and by $\alpha$ for fixed $d\geq d_0$). 
Then for any integer $n\geq 0$, $F^{\circ n}_{d,\alpha}(B)\subset \hat{\alpha}^n B$ ($\alpha^n B$ respectively).
For $n_0$ large enough $P_{\cvar_{n_0}}$ is contained in this ball $B$. 
For $n>2n_0$ we then set $\mathcal{T}_n=\hat{\alpha}^{n-2n_0} B$ ($\alpha^{n-2n_0} B$ respectively). 
The statements in the remark now follow quickly.

\subsection{The \texorpdfstring{$d\rightarrow \infty$ limit map}{The d → ∞ limit map}}
As mentioned above, an important tool in our approach is to analyze the maps $F_d$ as $d\to \infty.$ 
    Since $F_d(\mathbb{R}^{q-1})$ is bounded, it follows that as $d\to \infty$, $F_{d}(x_1,\ldots,x_{q-1})$ converges uniformly to the limit map 
\begin{equation}\label{eq:define F infty}
    F_{\infty}(x_1,\ldots,x_{q-1}),
\end{equation}
with coordinate functions
\begin{equation}\label{eq:def F infty coordinate}
    F_{\infty;i}(x_1,\ldots,x_{q-1}):=q\frac{1-e^{x_i}}{\sum_{j=1}^{q-1}e^{x_i}+1}.
\end{equation}
We write  $G_{\infty;i}(x_1,\ldots,x_{q-1}) = q\frac{1-x_i}{\sum_{j=1}^{q-1}x_j+1}$ for the $i$th coordinate function of the fractional linear map $G_{\infty}$. Note that $F_{\infty} = G_{\infty} \circ \exp$.

By Lemma \ref{lem:symmetry} for any $\pi \in S_q$, any $\mathbf{x} \in \mathbb{R}^{q-1}$ and any $d$ we have $\pi \cdot F_{d}(\mathbf{x}) =  F_{d}(\pi \cdot \mathbf{x})$. As the action of  $\pi$ on $\mathbb{R}^{q-1}$ does not depend on $d$, we immediately see $\pi \cdot F_{\infty}(\mathbf{x}) =  F_{\infty}(\pi \cdot \mathbf{x})$ follows.

In the next two sections we will prove Propositions~\ref{prop:convex} and~\ref{prop:2stepforward}. 
The idea is to first prove a variant of these propositions for the map $F_\infty$ and then use that $F_d\to F_\infty$ uniformly to finally prove the actual statements.
We use the description of $P_\cvar$ as intersection of half spaces $\pi\cdot H_{\geq -\cvar}$ in Section~\ref{sec:convex} and the description as the union of the $\pi \cdot D_\cvar$ in Section~\ref{sec:forward}.

\section{Convexity of the forward image of \texorpdfstring{$P_c$}{Convexity of the forward image of Pc}}\label{sec:convex}

This section is dedicated to proving Proposition \ref{prop:convex}.

Fix an integer $q\geq 3$.
For $\mu\in \mathbb{R}$ we define the half space $H_{\geq \mu}$ as in~\eqref{eq:half space}.
%\[
%H_{\geq \mu}=\left\{\mathbf{x}\in \mathbb{R}^{q-1}\mid \sum_{i=1}^{q-1} x_i\geq\mu\right\}.
%\]
The half space $H_{\leq \mu}$ is defined similarly. We denote by $H_{\mu}$ the affine space which is the boundary of $H_{\leq \mu}$.

In what follows we will often use that the map $G_\infty$ is a fractional linear transformation and thus preserves lines and hence maps convex sets to convex sets, see e.g.~\cite[Section 2.3]{boyd2004convex}.
\begin{lemma}\label{lem:forward convex infinity}
For all $\cvar>0$, the set $\exp(H_{\ge -\cvar}):=\{\exp(\mathbf{x})\mid \mathbf{x}\in H_{\geq -c}\}$ is strictly convex, consequently
\[
    G_\infty(\exp(H_{\ge -\cvar}))
\]
is strictly convex.
\end{lemma}
\begin{proof}
Since $G_\infty$ is a fractional linear transformation, it preserves convex sets. 
It therefore suffices to show that $\exp(H_{\geq -\cvar})$ is strictly convex. 

To this end take any $\mathbf{x},\mathbf{y}\in \exp(H_{\geq -\cvar})$ and let $\lambda\in (0,1)$.
We need to show that $\lambda \mathbf{x} +(1-\lambda)\mathbf{y}\in \exp(H_{\geq -\cvar})$.
By strict concavity of the logarithm we have
\[
\sum_{i=1}^{q-1}\log(\lambda x_i+(1-\lambda)y_i)\geq \sum_{i=1}^{q-1} \lambda \log(x_i)+(1-\lambda)\log(y_i)>-\cvar,
\]
we conclude that $\exp(H_{\geq -\cvar})$ is strictly convex.
\end{proof}

In what follows we need the \emph{angle} between the tangent space of $G_\infty(\exp(H_{-c}))$ for $c>0$ at $G_\infty(\mathbf{x})$ for any $\mathbf{x}\in \exp(H_{-\cvar})$ and the space $H_0$. This angle is defined as the angle of a normal vector of the tangent space pointing towards the interior of $G_\infty(\exp(H_{\geq -c}))$ and the vector $-\mathbf{1}$ (which is a normal vector of $H_0$).

\begin{lemma}\label{lem: forward infinite try 2}
For any $\cvar\in[0,q+1]$ and any $\mathbf{x}\in \exp(H_{-\cvar})$ the angle between the tangent space of $G_\infty(\exp(H_{-c}))$ at $G_\infty(\mathbf{x})$ and $H_0$ is strictly less than $\pi/2$.
\end{lemma}
\begin{proof}
We will first show that the tangent space cannot be orthogonal to $H_0$.

The map $G_\infty$ is invertible (when restricted to $\mathbb{R}_{>0}^{q-1}$) with inverse $G^{-1}_\infty$ whose coordinate functions are given by
\[
G^{-1}_{\infty,i}(y_1,\ldots,y_{q-1})= \frac{-qy_i}{\sum_{i=1}^{q-1}y_i+q}+1.
\]

Define $g:\mathbb{R}^{q-1}\setminus H_{-q}\to \mathbb{R}$ by $g(\mathbf{y})=\prod_{i=1}^{q-1}G^{-1}_{\infty,i}(\mathbf{y})$.
Then the image of $\exp(H_{-c})$ under $G_\infty$ is contained in the hypersurface $\{\mathbf{y}\in \mathbb{R}^{q-1}\mid g(\mathbf{y})=\exp(-c)\}.$
Therefore a normal vector of the tangent space of $G_\infty(\exp(H_{-c}))$ at $\mathbf{y}=G_\infty(\mathbf{x})$ is given by the gradient of the function $g$.
Thus to show that this tangent space is not orthogonal to $H_0$, we need to show that 
\begin{equation}\label{eq:not orthogonal}
\sum_{i=1}^{q-1} \tfrac{\partial}{\partial y_i}g(\mathbf{y})\neq 0.
\end{equation}
We have 
\begin{align*}
\sum_{i=1}^{q-1} \tfrac{\partial}{\partial y_i}g(\mathbf{y})&=\sum_{i=1}^{q-1}\sum_{j=1}^{q-1} \frac{\prod_{k=1}^{q-1}G^{-1}_{\infty,k}(\mathbf{y})}{G^{-1}_{\infty,j}(\mathbf{y})}\tfrac{\partial}{\partial y_i}G^{-1}_{\infty,j}(\mathbf{y})
\\
&=\sum_{j=1}^{q-1}\frac{\prod_{k=1}^{q-1}G^{-1}_{\infty,k}(\mathbf{y})}{G^{-1}_{\infty,j}(\mathbf{y})}\sum_{i=1}^{q-1}\tfrac{\partial}{\partial y_i}G^{-1}_{\infty,j}(\mathbf{y})
\\
&=\sum_{j=1}^{q-1}\frac{\prod_{k=1}^{q-1}G^{-1}_{\infty,k}(\mathbf{y})}{G^{-1}_{\infty,j}(\mathbf{y})} \cdot \frac{-q(\sum_{i=1}^{q-1}y_i+q)+q(q-1)y_j}{(\sum_{i=1}^{q-1}y_i+q)^2}
\\
&=\sum_{j=1}^{q-1}\frac{\prod_{k=1}^{q-1}G^{-1}_{\infty,k}(\mathbf{y})}{G^{-1}_{\infty,j}(\mathbf{y})}\cdot\frac{-(q-1)G_{\infty,j}^{-1}(\mathbf{y})-1}{\sum_{i=1}^{q-1}y_i+q}.
\end{align*}

Since $G^{-1}_{\infty,k}(\mathbf{y})>0$ for each $k$, all terms in the final sum are nonzero and have the same sign. This proves~\eqref{eq:not orthogonal}.

Since the angle between the tangent space of $G_\infty(\exp(H_{-c}))$ at $G_\infty(\mathbf{x})$ and $H_0$ depends continuously on $\mathbf{x}$ this angle should either be always less than $\pi/2$ or always be bigger.
Since by the previous lemma the set $G_\infty(\exp(H_{\geq -c}))$ is convex, it is the former.
% We note that $G_\infty(\exp(H_{\geq -\cvar}))$ is contained in the polytope
% $\text{conv}(\{-q\mathbf{e}_1,\ldots,-q\mathbf{e}_{q-1},q\mathbf{1}\})$.
% Indeed, this follows since $G_\infty$ preserves convex sets and the fact that $G_\infty(\mathbf{0})=q\mathbf{1}$ and $G_\infty(\infty\cdot \mathbf{e}_i)=-q\mathbf{e}_i$.
% From this it also follows the projection of $G_\infty(\exp(H_{\geq -c}))$ onto $H_0$ is contained in the projection of the interior of the simplex $\text{conv}(\{-q\mathbf{e}_1,\ldots,-q\mathbf{e}_{q-1}\})$ onto $H_0$.
% Therefore the tangent space of $G_\infty(\exp(H_{-\cvar}))$ at $G_
% \infty(\mathbf{x})$ makes an angle with $H_0$ that is strictly less than $\pi/2$ for any $\mathbf{x}\in \exp(H_{\leq -\cvar})$.
\end{proof}

We next continue with the finite case.
We will need the following definition.
The \emph{hypograph} of a function $f:D\to \mathbb{R}$ is the region $\{(x,y)\mid x\in D, y\leq f(x)\}$.
Below we will consider a hypersurface contained in $\mathbb{R}^{q-1}$ that we view as the graph of a function with domain contained in $H_0$. 
In this context the hypograph of such a function is again contained in $\mathbb{R}^{q-1}$, but the `positive $y$-axis' points in the direction of $\mathbf{1}$ as seen from $\mathbf{0} \in H_0$.

\begin{figure}
    \centering
\begin{tikzpicture}
\centering
   \filldraw[lightBlue](0.0,1.20298)--(0.0854683,1.22135)--(0.169623,1.23746)--(0.252203,1.25111)--(0.332949,1.26209)--(0.411607,1.27023)--(0.487928,1.27532)--(0.561674,1.27719)--(0.632616,1.27568)--(0.700539,1.27064)--(0.765244,1.26192)--(0.826546,1.2494)--(0.88428,1.23299)--(0.938302,1.21258)--(0.988486,1.18813)--(1.03473,1.15958)--(1.07696,1.12691)--(1.1151,1.09013)--(1.14914,1.04926)--(1.17907,1.00434)--(1.20488,0.955462)--(1.22663,0.902707)--(1.24437,0.846195)--(1.25817,0.786064)--(1.26815,0.722475)--(1.2744,0.655603)--(1.27707,0.585644)--(1.27631,0.512808)--(1.27227,0.437319)--(1.26513,0.359413)--(1.25508,0.279335)--(1.2423,0.197338)--(1.22699,0.113679)--(1.20935,0.0286192)--(-1.21547,-1.15836)--(-1.23236,-1.09062)--(-1.24685,-1.02198)--(-1.25874,-0.952488)--(-1.26784,-0.882209)--(-1.27397,-0.811205)--(-1.27693,-0.73954)--(-1.27657,-0.667276)--(-1.27272,-0.594473)--(-1.26524,-0.521195)--(-1.254,-0.447503)--(-1.2389,-0.373456)--(-1.21983,-0.299116)--(-1.19673,-0.224541)--(-1.16955,-0.149791)--(-1.13826,-0.0749248)--(-1.10284,0.0)--(-1.06333,0.0749248)--(-1.01976,0.149791)--(-0.972191,0.224541)--(-0.920715,0.299116)--(-0.865441,0.373456)--(-0.8065,0.447503)--(-0.744045,0.521195)--(-0.678246,0.594473)--(-0.609294,0.667276)--(-0.537392,0.73954)--(-0.462762,0.811205)--(-0.385635,0.882209)--(-0.306253,0.952488)--(-0.224868,1.02198)--(-0.141734,1.09062)--(-0.0571119,1.15836)--(1.18073,-0.0286192)--(1.11331,-0.113679)--(1.04496,-0.197338)--(0.975742,-0.279335)--(0.905718,-0.359413)--(0.83495,-0.437319)--(0.763498,-0.512808)--(0.691427,-0.585644)--(0.618797,-0.655603)--(0.545671,-0.722475)--(0.472109,-0.786064)--(0.398174,-0.846195)--(0.323925,-0.902707)--(0.249422,-0.955462)--(0.174724,-1.00434)--(0.0998897,-1.04926)--(0.0249778,-1.09013)--(-0.0499535,-1.12691)--(-0.124846,-1.15958)--(-0.199641,-1.18813)--(-0.274281,-1.21258)--(-0.348706,-1.23299)--(-0.422857,-1.2494)--(-0.496674,-1.26192)--(-0.570097,-1.27064)--(-0.643065,-1.27568)--(-0.715516,-1.27719)--(-0.787388,-1.27532)--(-0.858618,-1.27023)--(-0.929146,-1.26209)--(-0.998907,-1.25111)--(-1.06784,-1.23746)--(-1.13589,-1.22135)--(-1.20298,-1.20298)--(0.0,1.20298);
   \draw[-] (-3.8, 0) -- (2, 0) node[below] {$x_1$};
  \draw[-] (0, -2.5) -- (0, 2) node[left] {$x_2$};
  \draw[scale=1, domain=0:1, smooth, variable=\x, gR] plot ({20*ln(1 + (1 - exp(2*(1 - \x) - 4*\x))/(
   7*(6/7 + exp(2*(1 - \x) - 4*\x) + exp(-4 *(1 - \x) + 2*\x))))},{20*ln(1 + (1 - exp(-4*(1 - \x)+2*\x))/(
   7*(6/7 + exp(2*(1 - \x) - 4*\x) + exp(-4 *(1 - \x) + 2*\x))))});
  \draw[-,dY] (-3, 3) -- (3,-3) node[left, dY] {$H_0$};
  \draw[-,gG] (1.7,-1.7) -- (-1.7, 1.7) node[right, gG] {$\text{Dom}_\cvar$};
 \node [gB] at (-0.6,-0.5) {$F_\yvar(P_\cvar)$};
  \node [gR] at (1.8,1.8) {$F_\yvar(H_{-\cvar})$};
 
\end{tikzpicture}
    \caption{Depicting the situation in Lemma \ref{lem:forward finite}, for $q=3,\cvar = 2$ and $\yvar = \frac{1}{20}$. The domain $\text{Dom}_\cvar$ of the function $h_{\yvar,\cvar}$ which we define in the proof of Lemma \ref{lem:forward finite} is made by choosing $a' = -3$.}
    \label{fig:ConexityPicture}
\end{figure}
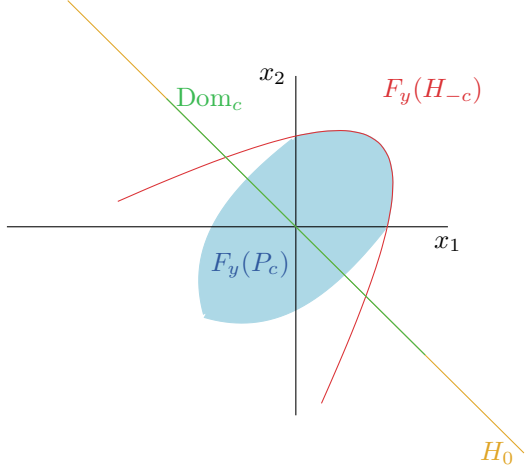

\begin{lemma}\label{lem:forward finite}
There exists $y_1>0$ such that for all $y\in[0,y_1)$ and $\cvar\in[0,q+1]$ the set $F_y(P_c)$ is contained in the hypograph of a concave function, $h_{\yvar,\cvar}$, with a convex compact domain in $H_0$.
\end{lemma}
\begin{proof}
We first prove that for any $\mathbf{x}\in H_{0}$ and $\cvar\in [0,q+1]$ there exists an open neighborhood $W_{\cvar,\mathbf{x}}=Y_{\cvar,\mathbf{x}}\times C_{\cvar,\mathbf{x}}\times X_{\cvar,\mathbf{x}}$ of $(0,\cvar,\mathbf{x})\in [0,1]\times[0,q+1] \times \mathbb{R}^{q-1}$ such that the following holds for any $(y',\cvar',\mathbf{x}')\in W_{\cvar,\mathbf{x}}$:
\begin{align}
   &\text{the angle between the tangent space of $F_{1/y'}(H_{-\cvar'})$ at $F_{1/y'}(\mathbf{x}'_{\cvar'})$ and $H_0$}\nonumber \\&\text{is strictly less than $\pi/2$,}\label{eq:90degrees}
\end{align}
where we denote $\mathbf{x}_\cvar:=\mathbf{x}-\frac{\cvar}{q-1}\mathbf{1}\in H_{-\cvar}$.
To see this note that by the previous lemma we have that the tangent space of  $F_{\infty}(H_{-\cvar})$ at $F_{\infty}(\mathbf{x}_\cvar)$ is not orthogonal to $H_0$ and in fact makes an angle of less than $\pi/2$ with $H_0$. 
Say it has angle $\pi/2-\gamma$. 
Since $(y,\cvar,\mathbf{x})\mapsto F_{1/y}(\mathbf{x}_\cvar)$ is analytic, there exists an open neighborhood $W_0$ of $(0,\cvar,\mathbf{x})$ such that for any $(y',\mathbf{x'},\cvar')\in W_0$ the angle between the tangent space of $F_{1/y'}(H_{-\cvar'})$ at $F_{1/y'}(\mathbf{x'}_{\cvar'})$ and $H_0$ is at most $\pi/2-\gamma/2$.
Clearly, $W_0$ contains an open neighborhood of $(0,\cvar,\mathbf{x})$ of the form $Y\times C\times X$ proving~\eqref{eq:90degrees}.

Next fix $\cvar\in [0,q+1]$ and $\mathbf{x}\in H_0$ and write $W_{\cvar,\mathbf{x}}=Y\times C\times X.$
Together with the implicit function theorem,~\eqref{eq:90degrees} now implies that for each $y'\in Y$ and any $\cvar'\in C$, that locally at $\mathbf{x}_{\cvar'}$, $F_{1/y'}( H_{-\cvar})$ is the graph of an analytic function $f_{y',\cvar',\mathbf{x}}$ on an open domain contained in $H_0$. 
Here we use that $F_{1/y}$ is invertible with analytic inverse.
By choosing $Y$ and $C$ small enough, we may by continuity assume that we have a common open domain, $D_{\cvar,\mathbf{x}}$, for these functions for all $\cvar'\in C$ and $y'\in Y$, where we may moreover assume that these functions are all defined on the closure of $D_{\cvar,\mathbf{x}}$.

We next claim, provided the neighbourhood $W=Y_{\cvar,\mathbf{x}}\times C_{\cvar,\mathbf{x}}$ is chosen small enough, that for each $y'\in Y$ and $\cvar'\in C$,
\begin{equation}\label{eq:hessian}
 \text{the largest eigenvalue of the Hessian $f_{y',\cvar',\mathbf{x}}$ on $D_{\cvar, \mathbf{x}}$ is strictly less than $0$.}
\end{equation}
To see this we note that by the previous lemma we know that $F_{\infty}(H_{\ge -\cvar})$ is strictly convex. 
Therefore the Hessian\footnote{Recall that the \emph{Hessian} of a function $f:U\to \mathbb{R}$ for an open set $U\subseteq\mathbb{R}^n$ at a point $u\in U$ is defined as the $n\times n$ matrix $H_f(u)$ with $(H_f(u))_{i,j}=\tfrac{\partial^2 f}{\partial x_i\partial x_j}(u)$. When these partial derivatives are continuous and the domain $U$ is convex, $f$ is concave if and only if its Hessian is negative definite at each point of the domain $U$~\cite{boyd2004convex}.} of $f_{0,\cvar,\mathbf{x}}$ on $D_{\cvar,\mathbf{x}}$ is negative definite, say its largest eigenvalue is $\delta<0$. 
Similarly as before, there exists an open neighborhood $W'\subseteq W$ of $(0,\cvar)$ of the form $W'=Y'\times C'$ such that for each $y'\in Y'$ and $\cvar'\in C'$, the function $f_{y',\cvar',\mathbf{x}}$ has a negative definite Hessian with largest eigenvalue at most $\delta/2<0$ for each $\mathbf{z}\in D_{\cvar,\mathbf{x}}$ (by compactness of the closure of 
$D_{\cvar,\mathbf{x}}$).
We now want to patch all these function to form a global function on a compact and convex domain.
We first collect some properties of $F_{1/y}$ that will allow us to define the domain.

First of all note that by compactness there exists $a>0$ such that for each $c\in [0,q+1]$, $\exp(P_c)\subset H_{\leq a}$ (where the inclusion is strict).
We now fix such a value of $a$.
Since $G_\infty$ is $S_q$-equivariant, we know that $G_\infty(H_{\leq a})=H_{\geq a'}$ for some $a'\in \mathbb{R}$.
We now choose $y^*>0$ small enough such that the following two inclusions hold for all $y\in [0,y^*]$ and $c\in [0,q+1]$
\begin{align}
    F_{1/y}(P_c)&\subset H_{\geq a'},\label{eq:image finite in Ha'}
    \\
    \text{proj}_{H_0}(F_\infty(H_{-c})\cap H_{\geq a'})&\subset  \text{proj}_{H_0}(F_{1/y}(H_{-c})), \label{eq:contained in image finite}
\end{align}
where $\text{proj}_{H_0}$ denotes the orthogonal projection onto the space $H_0$.
The first inclusion holds since $F_{1/y}$ converges uniformly to $F_\infty$ as $y\to 0.$
For the second inclusion note that  
\[
F_\infty(H_{-c})\cap H_{\geq a'}=G_\infty(\exp(H_{-c})\cap H_{\leq a})\subset F_{\infty}(H_{-c}).
\]
Because $\exp(H_{-c})\cap H_{\leq a}$ is compact, the desired conclusion follows since $F_{1/y}\to F_{\infty}$ uniformly as $y\to 0$.

Let us now consider for $\cvar\in [0,q+1]$ the projection 
\[
\text{Dom}_c:=\text{proj}_{H_0}(F_\infty(H_{-c})\cap H_{\geq a'}),
\]
see Figure \ref{fig:ConexityPicture}. Since $F_\infty(H_{-\cvar})\cap H_{\geq a'}$ is convex by Lemma~\ref{lem:forward convex infinity} and compact, it follows that $\text{Dom}_\cvar$ is compact and convex for each $\cvar\in [0,q+1]$.
Moreover, we claim that
\begin{equation}\label{eq:compact}
    \bigcup_{\cvar\in [0,q+1]} \left(\{\cvar\}\times \text{Dom}_\cvar\right) \subseteq [0,q+1]\times H_0 \text{ is compact.}
\end{equation}
Indeed, it is the continuous image of the compact set $\exp(H_{\geq -q-1})\cap H_{\leq a}$ under the map
\[
\exp(H_{\geq -q-1})\cap H_{\leq a}\to [0,q+1]\times H_0
\] 
defined by
\[
\mathbf{x}\mapsto \left(\sum_{i=1}^{q-1}x_i,\text{proj}_{H_0}(G_\infty(\mathbf{x}))\right).
\]

By \eqref{eq:contained in image finite} $\text{Dom}_c$ is contained in $\text{proj}_{H_0}(F_{1/y}(H_{-c}))$ for all $y\in [0,y^*]$ and $\cvar\in [0,q+1]$.
It follows that the sets 
$Y_{\cvar,\mathbf{x}}\times C_{\cvar,\mathbf{x}}\times D_{\cvar,\mathbf{x}}$, where $\mathbf{x}$ ranges over $H_{0}$ and $\cvar$ over $[0,q+1]$, form an open cover of $\{0\}\times \cup_{\cvar\in [0,q+1]}\left(\{\cvar\}\times \text{Dom}_c\right)$.
Since the latter set is compact by~\eqref{eq:compact}, we can take a finite sub cover. 
Therefore there exists $y_1>0$ such that for each $y\in [0,y_1)$ and each $\cvar\in[0,q+1]$ we obtain a unique global function $h_{y,\cvar}$ on the union of these finitely many domains, which by \eqref{eq:hessian} has a strictly negative definite Hessian.
By construction the union of these domains contains $\text{Dom}_\cvar$ for each $\cvar\in [0,q+1]$.
Consequently, restricted to $\text{Dom}_c$, $h_{y,\cvar}$ is a concave function for each $y\in [0,y_1)$ and $\cvar\in [0,q+1]$.
By~\eqref{eq:image finite in Ha'}, it follows that $F_{1/y}(P_c)$ is contained in the hypograph of $h_{y,\cvar}$, as desired.
\end{proof}

We can now finally prove Proposition~\ref{prop:convex}, which we restate here for convenience.
\begin{repprop}{prop:convex}
Let $q\geq 3$ be an integer. Then there exists $d_1>0$ such that for all $d\geq d_1$ and $c\in [0,q+1]$, $F_d(P_c)$ is convex.
\end{repprop}
\begin{proof}
By the previous lemma we conclude that for $d$ larger than $1/\yvar_1$, $F_d(P_\cvar)$ is contained in the hypograph of the function $h_{1/d,\cvar}$, denoted by $\text{hypo}(h_{\cvar,1/d})$ and moreover that this hypograph is convex, as the function $h_{1/d,\cvar}$ is concave on a convex domain.

Since $P_\cvar$ is invariant under the $S_q$-action, it follows that
\[
\exp(P_\cvar)=\bigcap_{\pi\in S_q} \pi \cdot (\exp(H_{\geq -\cvar })\cap H_{\leq a})
\]
and therefore by Lemma~\ref{lem:symmetry},
\begin{equation}\label{eq:intersection image}
F_d(P_c)=\bigcap_{\pi\in S_q}\pi \cdot (F_d(P_c))\subseteq \bigcap_{\pi\in S_q}\pi \cdot \text{hypo}(h_{1/d,\cvar}).
\end{equation}
We now claim that the final inclusion in~\eqref{eq:intersection image} is in fact an equality. 
To see the other inclusion, take some $\mathbf{z}\in \cap_{\pi\in S_q}\pi \cdot \text{hypo}(h_{1/d,\cvar})$.
By symmetry, we may assume that $\mathbf{z}$ is contained in $\mathbb{R}^{q-1}_{\geq 0}$.
Then $\mathbf{z}$ is equal to $F_d(\mathbf{x})$ for some $\mathbf{x}\in H_{\geq -c}\cap \mathbb{R}^{q-1}_{\leq 0}$, implying that $\mathbf{z}$ is indeed contained in $F_d(P_\cvar)$.

This then implies that $F_d(P_\cvar)$ is indeed convex being equal to the intersection of the convex sets $\pi \cdot \text{hypo}(h_{1/d,\cvar}).$
\end{proof}

\section{Forward invariance of \texorpdfstring{$P_c$ in two iterations}{Forward invariance of Pc in two iterations}}\label{sec:forward}

This section is dedicated to proving Proposition \ref{prop:2stepforward}.
We start with a version of the proposition for $d=\infty$ and after that consider finite $d$.

\subsection{Two iterations of \texorpdfstring{$F_\infty$}{Two iterations of F∞}}
Let $\Phi:\mathbb{R}^{q-1}\to \mathbb{R}^{q-1}$ be defined by
\[
    \Phi(x_1,\dots,x_{q-1})=F_\infty^{\circ 2}(x_1,\dots,x_{q-1})
\]
and its `restriction' to the half line $\mathbb{R}_{\leq/0}\cdot \mathbf{1}$, $\phi:\mathbb{R}_{\geq 0}\to\mathbb{R}_{\geq 0}$, by
\[
    \phi(t)=-\langle\Phi(-t/(q-1)\cdot \mathbf{1}),\mathbf{1}\rangle,
\]
where we use $\langle \cdot,\cdot\rangle$ to denote the standard inner product on $\mathbb{R}^{q-1}.$

This subsection is devoted to proving the following result.
\begin{prop}\label{prop:2step-forward infinity}
For any $\cvar\ge 0$ we have
\[
\Phi(P_{\cvar})\subseteq P_{\phi(\cvar)}\subsetneq P_\cvar.
\]
\end{prop}
By the definition of $P_c$ in terms of $D_c$, \eqref{eq:fund domain}, and the $S_q$-equivariance of the map $F_\infty$ and hence of the map $\Phi$, it suffices to prove this for $P_c$ replaced by $D_c$.
This can be derived from the following two statements:
\begin{enumerate}
    \item[(i)] For any $\cvar\ge 0$ the minimum of $\langle\Phi(\mathbf{x}),\mathbf{1}\rangle$ on $-\cvar\Delta$ is attained at $-\cvar/(q-1)\cdot \mathbf{1}$.
    \item[(ii)] For any $\cvar> 0$ we have $\phi(\cvar)<\cvar$.
\end{enumerate}
Indeed, these statements imply that for any $\cvar> 0$ we have that $\Phi(-\cvar\Delta)\subseteq D_{\phi(\cvar)}\subsetneq D_\cvar$.
Clearly this is sufficient, since $D_\cvar=\cup_{0\le \cvar'\le \cvar} -\cvar'\Delta$ and therefore
\[
    \Phi(D_\cvar)=\cup_{0\le \cvar'\le \cvar}\Phi(-\cvar'\Delta)\subseteq \cup_{0\le \cvar'\le \cvar}D_{\phi(\cvar')}\subseteq D_{\phi(\cvar)}\subsetneq D_{\cvar}.
\]
We next prove both statements, starting with the first one.
\subsubsection{Statement (i)}
\begin{prop}\label{prop:infinite2iteration}
Let $\cvar\geq 0$. Then for any $\mathbf{x}\in -\cvar \Delta$ we have that
\[
    \langle\Phi(\mathbf{x}),\mathbf{1}\rangle \ge \left\langle\Phi\left(\frac{-\cvar}{q-1}\mathbf{1}\right),\mathbf{1}\right\rangle.
\]
Moreover, equality happens only at $\mathbf{x}=\frac{-\cvar}{q-1}\mathbf{1}$.
\end{prop}

Before giving a proof, let us fix some further notation. 
By definition we have
\[
    \langle \Phi(\mathbf{x}),\mathbf{1}\rangle =\sum_{i=1}^{q-1}q\frac{1-e^{F_{\infty;i}(\mathbf{x})}}{\sum_{j=1}^{q-1}e^{F_{\infty;j}(\mathbf{x})}+1}=\frac{q^2}{{\sum_{j=1}^{q-1}e^{F_{\infty;j}(\mathbf{x})}+1}}-q,
\]
where we recall that $F_{\infty;j}$ denotes the $j$th coordinate function of $F_\infty$.
Thus the $i$th coordinate of the gradient of $\langle \Phi(\mathbf{x}),\mathbf{1}\rangle$ is given by
\begin{align*}
    \psi_i(\mathbf{x})&:=\frac{-q^2}{\left({\sum_{j=1}^{q-1}e^{F_{\infty;j}(\mathbf{x})}+1}\right)^2}\left(\sum_{j=1}^{q-1}e^{F_{\infty;j}(\mathbf{x})}\cdot \frac{\partial F_{\infty;j}}{\partial x_i}(\mathbf{x})\right)\\
    &=\frac{q^3e^{x_i}\left(e^{F_{\infty;i}(\mathbf{x})}(1+\sum_{j=1}^{q-1}e^{x_j})+\sum_{j=1}^{q-1}e^{F_{\infty;j}(\mathbf{x})}(1-e^{x_j})\right)}{\left({\sum_{j=1}^{q-1}e^{F_{\infty;j}(\mathbf{x})}+1}\right)^2\left(\sum_{j=1}^{q-1}e^{x_j}+1\right)^2}.%\times
\end{align*}

Let us define the following functions $v_i:\mathbb{R}^{q-1}\to\mathbb{R}$ for $i=1,\dots,q-1$ as
\[
    v_i(\mathbf{x}):=x_i \left(e^{G_{i}} (1+ \sum_{\substack{j=1}}^{q-1} x_j) + \sum_{\substack{j=1}}^{q-1} e^{G_{j}} (1 - x_j)  \right),
\]
where we write
\[
    G_i:=G_{\infty;i}(\mathbf{x}) = \frac{q (1-x_i)}{1 + x_1 + \cdots + x_{q-1}}.
\]

Then we see that 
\[
    \psi_i(\mathbf{x})=\frac{q^3}{\left({\sum_{j=1}^{q-1}e^{F_{\infty;j}(\mathbf{x})}+1}\right)^2\left(\sum_{j=1}^{q-1}e^{x_j}+1\right)^2}\cdot v_i(e^{x_1},\dots,e^{x_{q-1}}),
\]
and $\psi_1(\mathbf{x})=\dots =\psi_{q-1}(\mathbf{x})$ if and only if $v_1(\exp(\mathbf{x}))=\dots=v_{q-1}(\exp(\mathbf{x}))$.

\begin{proof}[Proof of Proposition~\ref{prop:infinite2iteration}]
First of all observe that the function $\langle\Phi(\mathbf{x}),\mathbf{1}\rangle$ is invariant under the permutation of the coordinates of $\mathbf{x}$. Thus we can assume that 
\[
    \mathbf{x}\in U:=\{\mathbf{y}\in\mathbb{R}^{q-1}~|~0\ge y_1 \dots\ge y_{q-1}  \}
\]
 and not all the coordinates of $\mathbf{x}$ are equal.

Now it is enough to show that there exists a vector $\mathbf{0}\neq \mathbf{w}\in\mathbb{R}^{q-1}$ such that in the direction of $\mathbf{w}$ the function is (strictly) decreasing, $\langle\mathbf{w},\mathbf{1}\rangle=0$ and $\mathbf{x}+t_0\mathbf{w}\in U$ for some small $t_0>0$.
Let \[\ell=\min\{1\le i\le q-2 ~|~x_i>x_{i+1}\},\] which is finite, since not all of the coordinates of $\mathbf{x}$ are equal. 

We claim that $\mathbf{w}=-\frac{\mathbf{e}_1+\dots+\mathbf{e}_{\ell}}{\ell}+\mathbf{e}_{\ell+1}$ satisfies the desired conditions. 
Clearly, $\mathbf{w}$ is perpendicular to $\mathbf{1}$ and $\mathbf{x}+t\mathbf{w}\in U$ for $t$ small enough. 
Now let us calculate the derivative of 
\[
    g(t):=\langle\Phi(\mathbf{x}+t\mathbf{w}),\mathbf{1}\rangle.
\]
Using the notation defined above, we obtain
\begin{align*}
    g'(0)=&-\frac{\psi_1(\mathbf{x})+\dots+\psi_{\ell}(\mathbf{x})}{\ell}+\psi_{\ell+1}(\mathbf{x})\\
    &= -\psi_{\ell}(\mathbf{x})+\psi_{\ell+1}(\mathbf{x})\\
    &=-C\cdot (v_{\ell}(\exp(\mathbf{x}))-v_{\ell+1}(\exp(\mathbf{x})))
    \\
    &=-C\cdot (v_{\ell}(\mathbf{y})-v_{\ell+1}(\mathbf{y})),
\end{align*}
where $C>0$ and  $\mathbf{y}=\exp(\mathbf{x})$. 
In particular, \[1\ge y_1=y_2=\ldots=y_\ell>y_{\ell+1}\ge \ldots \ge y_{q-1}\ge 0.\]
So to conclude that $g'(0)<0$ and finish the proof, we need to show that \begin{equation}\label{eq:condition on y}
    v_\ell(\mathbf{y})-v_{\ell+1}(\mathbf{y})>0.
\end{equation}
Lemma~\ref{lemma:balancing} shows that we may assume $\mathbf{y}$ satisfies $1\ge y_1=y_2=\ldots=y_\ell>y_{\ell+1}\ge y_{\ell+2}=\ldots =y_{q-1}\ge 0$.
Lemma~\ref{lemma:3variable} below shows that for those vectors $\mathbf{y}$ \eqref{eq:condition on y} is indeed true.
So by combining Lemma~\ref{lemma:balancing} and Lemma~\ref{lemma:3variable} below we obtain~\eqref{eq:condition on y} and finish the proof.
\end{proof}

\begin{lemma}\label{lemma:balancing}
If $1\ge y_1=y_2\ldots=y_\ell>y_{\ell+1}\ge \ldots\ge y_{q-1}\ge 0$ for some $1\le \ell \le q-2$, then 
\[
    v_\ell(\mathbf{y})-v_{\ell+1}(\mathbf{y})\ge v_\ell(\mathbf{x})-v_{\ell+1}(\mathbf{x}),
\]
where $\mathbf{x}\in\mathbb{R}^{q-1}$ is defined as 
\[
    x_j=\left\{\begin{array}{cc}y_j&\textrm{ if $j\le  \ell+1$}\\ \frac{y_{\ell+2}+\dots+y_{q-1}}{q-\ell-2}&\textrm{ if $j>\ell+1$}\end{array}\right.
\]
for $1\le j\le q-1$.
\end{lemma}
\begin{proof}
By continuity, it suffices to show 
\begin{equation}\label{eq:condition yprime}
    v_\ell(\mathbf{y})-v_{\ell+1}(\mathbf{y})\ge v_\ell(\mathbf{x})-v_{\ell+1}(\mathbf{x}),
\end{equation}
where $\mathbf{x}\in\mathbb{R}^{q-1}$ is defined as
\[
    x_j=\left\{\begin{array}{cc}y_j&\textrm{ if $j\neq i,i+1$}\\ \frac{y_i+y_{i+1}}{2}&\textrm{ if $j=i$ or $j=i+1$}\end{array}\right.
\]
for $1\le j\le q-1$ and any $i\geq \ell+2.$

For $t\in \mathbb{R}$ we define $\mathbf{y}(t)$ by 
\[
    y_j(t):=\left\{\begin{array}{cc}
        y_{j} & \textrm{if $j\neq i,i+1$}\\
        y_i-t & \textrm{if $j=i$}\\
        y_{i+1}+t & \textrm{if $j=i+1$}
    \end{array}\right.
\]
for $j=1,\ldots,q-1$.
Note that $\mathbf{y}(0)=\mathbf{y}$ and $\mathbf{y}(y_i/2-y_{i+1}/2)=\mathbf{x}$.
We further define
\begin{align*}
    \Delta(t):=&v_\ell(\mathbf{y}(t))-v_{\ell+1}(\mathbf{y}(t)).
\end{align*}

After a straightforward calculation we can express $\Delta(t)$ as 
\begin{align*}
    \Delta(t)&=y_\ell e^{G_\ell}(1+\sum_{j\ge 1}^{q-1} y_j)-y_{\ell+1}e^{G_{\ell+1}}(1+\sum_{j\ge 1}^{q-1} y_j)\\
    &+y_\ell\sum_{j\neq i,i+1}e^{G_j}(1-y_j)-y_{\ell+1}\sum_{j\neq i,i+1}e^{G_j}(1-y_j)\\
    &+ (y_\ell-y_{\ell+1})\left(e^{G_i(t)}(1-y_i+t)+e^{G_{i+1}(t)}(1-y_{i+1}-t)\right),
\end{align*}
where we write
$
    G_\ell:=G_{\infty;\ell}(\mathbf{y}(t)) = \frac{q (1-y_\ell)}{1 + y_1 + \cdots + y_{q-1}},
$
for $\ell\not \in \{ i,i+1\}$ and we write $G_\ell(t) = G_{\infty;\ell}(\mathbf{y}(t))$ when $\ell \in \{i,i+1\}$. This notation indicates that $G_\ell$ is a constant function of $t$ when $\ell\not \in \{ i,i+1\}$.
Now observe that the function appearing in the last row,
\[
    g(t):=e^{G_i(t)}(1-y_i+t)+e^{G_{i+1}(t)}(1-y_{i+1}-t),
\]
is convex on $t\in[0,y_i-y_{i+1}]$, since its second derivative is given by
\begin{align*}
    g''(t)&=e^{G_i(t)}\frac{(1-y_i+t)q^2}{(1+y_1+\dots+y_{q-1})^2}+2e^{G_i(t)}\frac{q}{1+y_1+\dots+y_{q-1}}\\
    &+e^{G_{i+1}(t)}\frac{(1-y_{i+1}-t)q^2}{(1+y_1+\dots+y_{q-1})^2}+2e^{G_{i+1}(t)}\frac{q}{1+y_1+\dots+y_{q-1}}>0.
\end{align*}
As $g(t)=g(y_{i}-y_{i+1}-t)$, we obtain that $g(t)$ has a unique minimizer in $[0,y_{i}-y_{i+1}]$ exactly at $t$ such that $=y_{i}-y_{i+1}-t$. In other words,
\[
t=\frac{y_{i}-x_{i+1}}{2}   
\]
is the unique minimizer of $g(t)$ on this interval and thus for $\Delta(t)$. 
This implies \eqref{eq:condition yprime} and hence the lemma.
\end{proof}

\begin{lemma}\label{lemma:3variable}
    Let $1 \geq x_1 > x_2 \geq x_3 \geq 0$ and $q-2 \geq l \geq 1$. Then 
    \[
        v_{l}(\underbrace{x_1, \cdots, x_1}_{l}, x_2, \underbrace{x_3, \cdots, x_3}_{q-l-2}) > 
        v_{l+1}(\underbrace{x_1, \cdots, x_1}_{l}, x_2, \underbrace{x_3, \cdots, x_3}_{q-l-2}).
    \]
\end{lemma}

\begin{proof}
The algebraic manipulations that are done in this proof, while elementary, involve quite large expressions. Therefore we have supplied additional Mathematica code in Appendix~\ref{sec: Mathematica code} that can be used to verify the computations. We define
    \begin{align*}
        \Delta(y_1,y_2,y_3;t):= 
        &\left(y_1
        y_3 (t-l-1)+(l+1) y_1+(l+1) y_1 y_2-l y_2\right) e^{A_{1}(y_1,y_2,y_3;t)}+\\
        &\left(-y_2 y_3 (t-l-1)-(l+1) y_1 y_2+y_1-2 y_2\right) e^{A_{2}(y_1,y_2,y_3;t)} +\\
           &\left(y_1-y_2\right) \left(1-y_3\right) (t-l-1) e^{A_{3}(y_1,y_2,y_3;t)},
  \end{align*}
  where 
  \[
        A_{i}(y_1,y_2,y_3;t): = \frac{(t+1)(1-y_i)}{1+ly_1 + y_2 + (t-(l+1))y_3}
  \]
  for $i = 1,2,3$ (see Listing~\ref{code: Ai Delta}). One can check that
  \[
    \Delta(x_1,x_2,x_3;q-1) = v_{l}(x_1, \cdots, x_1, x_2, x_3, \cdots, x_3) - v_{l+1}(x_1, \cdots, x_1, x_2, x_3, \cdots, x_3).
  \]
  We will treat $t$ as a variable and vary it while keeping the values that appear in the exponents constant. To that effect let $C_i = A_i(x_1,x_2,x_3;q-1)$ and define 
  \begin{align*}
    y_1(t) &= \frac{C_1 (l-t-1)+C_3 (t-l-1)+C_2+t+1}{C_3 (t-l-1)+C_1 l+C_2+t+1},\\
    y_2(t) &= \frac{C_3 (t-l-1)+C_1 l-C_2 t+t+1}{C_3 (t-l-1)+C_1
   l+C_2+t+1}, \\
    y_3(t) &= \frac{C_1 l-C_3 (l+2)+C_2+t+1}{C_3 (t-l-1)+C_1 l+C_2+t+1}.
  \end{align*}
  These values are chosen such that for $t_0=q-1$ we have $y_i(t_0) = x_i$ and $A_i(y_1(t),y_2(t),y_3(t);t) = C_i$ independently of $t$ for $i = 1,2,3$ (see Listings~\ref{code: yi}~and~\ref{code: yi2}).
  Therefore $\Delta(y_1(t),y_2(t),y_3(t);t)$ is a rational function of $t$ and we want to show that it is positive at $t = q-1$. We can explicitly calculate that 
  \[
    \Delta(y_1(t),y_2(t),y_3(t);t) = 
    \left(\frac{1+t}{C_3 (t-l-1)+C_1 l+C_2+t+1}\right)^2 \cdot r(t),
  \]
  where $r$ is a linear function (see Listing~\ref{code: r}). It is thus enough to show that $r(q-1) > 0$. We will do this by showing that $r(l+1) > 0$ and that the slope of $r$ is positive.
  We find that $r(l+1)$ is equal to
  \[
    r(l+1)=u_1\cdot e^{C_1}+u_2\cdot e^{C_2},
  \]
  where
  \begin{align*}
      u_1&=2+l+C_2-2C_1+lC_1C_2-lC_1^2\\
      u_2&=-\left(2+l+lC_1-(l+1)C_2+C_1C_2-C_2^2\right).
  \end{align*}
  This is part of the output of Listing~\ref{code: r slope}.
  Note that by construction, since $1\geq x_1>x_2\geq x_3$, we have $0\leq C_1<C_2\leq C_3.$
Therefore the sum of the coefficients of $e^{C_1}$ and $e^{C_2}$ satisfies
  \begin{align*}
      u_1+u_2&=(l+2)(C_2-C_1)+(l-1)C_1C_2-lC_1^2+C_2^2\\
      &=(l+2+C_2+lC_1)(C_2-C_1)>0.
  \end{align*}
  Now we will separate two cases depending on the sign of  the coefficient of $u_2$. If $u_2$ is non-negative, then
  \begin{align*}
    r(l+1)&=u_1e^{C_1}+u_2e^{C_2}\ge u_1e^{C_1}+u_2e^{C_1}=(u_1+u_2)e^{C_1}>0.
  \end{align*}
  If $u_2$ is negative, then 
  \[
    2+(1+C_1-C_2)l>C_2-C_1C_2+C_2^2=(1+C_2-C_1)C_2.
  \] 
  In particular $2+(1+C_1-C_2)l>0$. Thus
  \begin{align*}
      r(l+1)&=e^{C_2}(u_1e^{C_1-C_2}-u_2)\\
      &\ge (1+C_1-C_2)u_1-u_2= C_1(C_2-C_1)(2+(1+C_1-C_2)l)>0.
  \end{align*}
  The slope of $r$ is given by 
  \[
    s:=\left(1+ C_3 - C_1\right)e^{C_1} - (1 + C_3 - C_2) e^{C_2} + (C_2-C_1) C_3 e^{C_3}. 
  \]
  This is part of the output of Listing~\ref{code: r slope}.
  To show that this is positive we show that $s \cdot e^{-C_2}$ is positive. Because both $1+ C_3 - C_1$ and $C_2 - C_1$ are positive we find
  \begin{align*}
    s \cdot e^{-C_2} &= \left(1+ C_3 - C_1\right)e^{C_1-C_2} - (1 + C_3 - C_2)+ (C_2-C_1) C_3 e^{C_3-C_2} \\
    &\geq 
    \left(1+ C_3 - C_1\right)(1+C_1-C_2) - (1 + C_3 - C_2)+ (C_2-C_1) C_3 (1+C_3-C_2)\\
    &= 
    (C_2-C_1)(C_1 + C_3 (C_3 - C_2)),
  \end{align*}
  which is positive because $0\leq C_1 < C_2 \leq C_3$. This concludes the proof.
\end{proof}

We now continue with the second statement.
\subsubsection{Statement (ii)}
\begin{prop}
For any $x>0$ we have that
\[
    \left\langle\Phi\left(\frac{-x}{q-1}\mathbf{1}\right),\mathbf{1}\right\rangle > -x.
\]
\end{prop}
\begin{proof}
The statement is equivalent to
\[
    \phi(x)<x.
\]
for $x>0$.
By definition we know that
\[
    \phi(x)=(q-1)\frac{q(e^{f(x)}-1)}{(q-1)e^{f(x)}+1},
\]
where 
\[
    f(x)=-q\frac{e^{-x/(q-1)}-1}{(q-1)e^{-x/(q-1)}+1}.
\]

First note that $\phi(\mathbb{R}_{>0})\subseteq (0,q)$. This means that if $x\ge q$, the statement holds. 
Thus we can assume that $0<x<q$. Now, the inequality $\phi(x)<x$ can be written as
\[
    e^{f(x)}<\frac{x+q(q-1)}{(q-1)(q-x)},
\]
because $q-x>0$. 
By taking logarithm of both sides, we see that $\phi(x)<x$ is equivalent to
\[
    -q\frac{e^{-x/(q-1)}-1}{(q-1)e^{-x/(q-1)}+1}<\log\left(\frac{x+q(q-1)}{(q-1)(q-x)}\right).
\]
Since $\frac{x+q(q-1)}{(q-1)(q-x)}> \frac{0+q(q-1)}{(q-1)q}\ge 1 $,
we can use the inequality $\log(b)>2\frac{b-1}{b+1}$ for $b=\frac{x+q(q-1)}{(q-1)(q-x)}$. Therefore, to show $\phi(x)<x$, it is sufficient to prove that
\[
    -q\frac{e^{-x/(q-1)}-1}{(q-1)e^{-x/(q-1)}+1}\le \frac{-2qx}{(q-2)x-2q(q-1)},
\]
or, equivalently
\[
    (2q-2-x)\le (x+2q-2)e^{-x/(q-1)}.
\]
This follows from the fact that $g(t)=(t+2q-2)e^{-t/(q-1)}-(2q-2-t)$ is a convex function on $\mathbb{R}_{\ge 0}$, its derivative satisfies $g'(0)=0$ and $g(0)=0$.
This concludes the proof.
\end{proof}

\subsection{Two iterations of \texorpdfstring{$F_d$}{Two iterations of Fd}}

As before, we view $\yvar=1/d$ as a continuous variable. 
Let us define $\Phi: \mathbb{R}^{q-1}  \times [0,\frac{1}{2}] \rightarrow \mathbb{R}^{q-1}$ by 
\[
\Phi(x_1,\ldots,x_{q-1},\yvar) = F_{1/\yvar}^{\circ 2} (x_1,\ldots,x_{q-1}).
\]
Note that this map is analytic in all its variables. 
For simplicity, if $\yvar^*$ is fixed, then we use the notation $\Phi_{\yvar^*}(x_1,\ldots,x_{q-1})$ for $\Phi(x_1,\ldots,x_{q-1},\yvar)|_{\yvar=\yvar^*}$, and if $\yvar=0$, then $\Phi(x_1,\ldots,x_{q-1}):=\Phi_0(x_1,\ldots,x_{q-1})$.
\begin{lemma}
There exist positive constants $A>0$ and $\cvar_0>0$, such that for any $0<\cvar\leq\cvar_0$ we have
\[
   \cvar-\phi(\cvar) \ge A\cvar^3.
\]
\end{lemma}
\begin{proof}
By definition we know that
\[
    \phi(x)=(g\circ f)(x)=(q-1)\frac{q(e^{f(x)}-1)}{(q-1)e^{f(x)}+1},
\]
where 
\begin{align*}
    f(x)&=-q\frac{e^{-x/(q-1)}-1}{(q-1)e^{-x/(q-1)}+1},\\
    g(x)&=(q-1)q\frac{e^x-1}{(q-1)e^x+1}.
\end{align*}

Let us calculate the Taylor expansion of $f(x)$ and $g(x)$ around 0:
\begin{align*}
    f(x)&=\frac{1}{q-1}x+\frac{q-2}{2(q-1)^2q}x^2+\frac{(q^2-6q+6)}{6(q-1)^3q^2}x^3+O(x^4),\\
    g(x)&=(q-1)x-\frac{(q-1)(q-2)}{2q}x^2+\frac{(q-1)(q^2-6q+6)}{6q^2}x^3+O(x^4).
\end{align*}
Thus their composition has the following Taylor expansion around $0$:
\begin{align*}
    (g\circ f)(x)&=x-\frac{1}{6(q-1)^2}x^3+O(x^4).
\end{align*}
This implies that there exists $\cvar_0>0$ and $A>0$, such that for any $\cvar_0\ge x\ge 0$ we have
\[
    x-\phi(x)\ge Ax^3,
\]
as desired.
\end{proof}

The next proposition implies forward invariance of $P_c$ under $F_d^{\circ 2}$ for $c$ small enough and $d$ large enough. 
\begin{prop}\label{prop:forwardunder2iterates q=3 w=w_c}
There exists $\cvar_0>0$ and $d_0>0$. Such that for all $\cvar\in (0,\cvar_0]$ and integers $d\geq d_0$ there exists $0<\cvar'<\cvar$ such that
\[F_d^{\circ 2}(\Dvar_\cvar) \subset \Dvar_{\cvar'}.
\]
\end{prop}
\begin{proof}
By the previous lemma we know that there is a $\cvar'_0>0$ and an $A>0$, such that  for any $\cvar\le \cvar'_0$ we have 
\[
    \|\Phi(-\cvar/(q-1)\cdot \mathbf{1})+\cvar/(q-1)\cdot \mathbf{1}\|\ge A\cvar^3.
\]
Here we denote by $\|\mathbf{x} \|=\left(\sum_{i=1}^{q-1}x_i^2\right)^{1/2}$, the standard $2$-norm on $\mathbb{R}^{q-1}.$
By Proposition~\ref{prop:infinite2iteration}, we have that for any $\mathbf{x}\in \Dvar_\cvar$, $\Phi(\mathbf{x})$ is contained in $D_{\phi(\cvar)}$. 
Therefore, denoting by $B_r(y)$ the ball of radius $r$ around $y$,

\begin{equation}\label{eq:strict inequality 2-forward}
 B_{A\cvar^{3}/2}(\Phi(\mathbf{x}))\cap (-\infty,0]^{q-1}\subseteq \Dvar_{\phi(\cvar)+A\cvar^3/2} \subsetneq  \Dvar_\cvar.
 \end{equation}

Now let us consider the Taylor approximation of $\Phi_{\yvar}(x_1,\ldots,x_{q-1})$ at $\mathbf{0} = (0,\ldots,0)$. 
Since for any $\yvar^*\in[0,1]$ the map $F_{1/y^*}(x_1,\ldots,x_{q-1})$ has $\mathbf{0}$ as a fixed point of derivative $-\textrm{Id}$, there exists constants $\cvar_1,C_1\ge 0$ such that for any $\yvar \in[0,1]$  and $\mathbf{x} = (x_1,\ldots,x_{q-1})\in [-\cvar_1,0]^{q-1}$  we have
\[
    \|\Phi_{\yvar}(\mathbf{x}) - \textrm{Id}(\mathbf{x})-T_{3,\yvar}(\mathbf{x})\|\le C_1\cdot \|\mathbf{x}\|^4,
\]
where $\textrm{Id}(\mathbf{x})+T_{3,\yvar}(\mathbf{x})$ is the 3rd order Taylor approximation of $\Phi_{\yvar}(\mathbf{x})$ at $\mathbf{0}$. 
Note that the second order term is equal to $0$ because the derivative of $F_{1/y^*}(x_1,\ldots,x_{q-1})$ at $\mathbf{0}$ equals $-\textrm{Id}$.
In particular, $T_{3,\yvar}(\mathbf{x})=T_\yvar((\mathbf{x}),(\mathbf{x}),(\mathbf{x}))$ for some multi-linear map $T_\yvar \in\textrm{Mult}((\mathbb{R}^{q-1})^3,\mathbb{R}^{q-1})$, and as $\yvar \to 0$ the map $T_{3,\yvar}$ converges uniformly on $[-q,0]^{q-1}$ to  $T_{3,0}$. Specifically, for any $\mathbf{x} = (x_1,\ldots,x_{q-1})\in[-\cvar_1,0]^{q-1}$
\[
    \|T_{3,\yvar}(\mathbf{x})-T_{3,0}(\mathbf{x})\|\le A_3(\yvar)\|\mathbf{x}\|^3
\]
for some function $A_3$ that satisfies $\lim_{\yvar \to 0} A_3(\yvar)=0$.

Putting this together and making use of the triangle inequality, we obtain that for any $0<\cvar\leq \min\{\cvar_1,\cvar'_0\}$ and any $\mathbf{x} = (x_1,\ldots,x_{q-1})\in\Dvar_\cvar$
\begin{align*}
    \|\Phi_{\yvar}(\mathbf{x})-\Phi(\mathbf{x})\|&\le\|\Phi_{\yvar}(\mathbf{x})-\textrm{Id}(\mathbf{x})-T_{3,\yvar}(\mathbf{x})\|\\
    &+\|\textrm{Id}(\mathbf{x})+T_{3,\yvar}(\mathbf{x})-\textrm{Id}(\mathbf{x})-T_{3,0}(\mathbf{x})\|\\
    &+\|\textrm{Id}(\mathbf{x})+T_{3,0}(\mathbf{x})- \Phi(\mathbf{x})\| \\
    &\le 2C_1\|x\|^4+A_3(y))\|x\|^3\leq K(2C_1c+A_3(y))c^3,
\end{align*}
for some constant $K>0$ (using that the $2$-norm and the $1$-norm are equivalent on $\mathbb{R}^{q-1}$.)
Now let us fix $0<\cvar_0\leq \min \{\cvar_1,\cvar_0'\}$ small enough such that $K2C_1\cvar_0<A/4$ and fix a $\yvar_0>0$ such that for any any $0\le \yvar\le \yvar_0$ we have $K A_3(\yvar)\le A/4$.

Then by \eqref{eq:strict inequality 2-forward}, for any $0\le \yvar \le \yvar_0$, $0\le \cvar \le \cvar_0$ and $\mathbf{x} = (x_1,\ldots,x_{q-1})\in \Dvar_\cvar$, 
\begin{align*}
    \Phi_\yvar(\Dvar_\cvar)\subseteq B_{Ac^3/2}(\Phi(\Dvar_\cvar))\cap (-\infty,0]^{q-1} \subseteq \Dvar_{\phi(\cvar)+A\cvar^3/2}\subsetneq \Dvar_\cvar.
\end{align*}
So we can take $\cvar'=\phi(\cvar)+A\cvar^3/2$.
\end{proof}

\subsection{Proof of Proposition~\ref{prop:2stepforward}}
We are now ready to prove Proposition~\ref{prop:2stepforward}, which we restate here for convenience. 

\begin{repprop}{prop:2stepforward}
Let $q\geq 3$ be an integer. There exists $d_2>0$ such that for all integers $d\geq d_2$ the following holds:
for any $\cvar\in(0,q+1]$ there exists $0<\cvar'<\cvar$ such that
\[
F_d^{\circ 2}(P_\cvar)\subset P_{\cvar'}.
\]
\end{repprop}
\begin{proof}
We know by Proposition~\ref{prop:forwardunder2iterates q=3 w=w_c} there is a $d_0>0$ and a $\cvar_0 >0$ such that for $d \geq d_0$ and $\cvar \in (0, \cvar_0)$ there exist $\cvar'<\cvar$ such that $F_d^{\circ 2}(D_\cvar) \subset D_{\cvar'}$.
As $P_\cvar = \cup_{\pi \in S_q} \pi \cdot D_\cvar$, we see by Lemma \ref{lem:symmetry} that for  $d \geq d_0$ and $\cvar \in (0, \cvar_0)$ we have $F_d^{\circ 2}(P_\cvar) \subset P_{\cvar'}$.

Next we consider $\cvar \in [\cvar_0, q+1]$. 
By Proposition~\ref{prop:2step-forward infinity} we know  $F_\infty^{\circ 2}(P_\cvar) \subset P_{\phi(\cvar)}$ and $\phi(\cvar)<\cvar$ for any $\cvar>0$. 
As $F_d$ converges to $F_\infty$ uniformly, we see for each $\cvar \in [\cvar_0, q+1]$ there is a $d_\cvar>0$ large enough such that for $d \geq d_\cvar$ and $\cvar'=\cvar/2+\phi(\cvar)/2$ we have $F_d^{\circ 2}(P_{\hat{\cvar}}) \subsetneq P_{\cvar'}$ for all $\hat{\cvar}$ sufficiently close to $\cvar$. 
By compactness of $[\cvar_0, q+1]$, we obtain that there is a $d_{\text{max}} >0$ such that for any $d> d_{\text{max}}$ and any $\cvar \in [\cvar_0, q+1]$ there exists $\cvar'<\cvar$ such that $F_d^{\circ 2}(P_\cvar) \subsetneq P_{\cvar'}$.
The proposition now follows by taking $d_2 = \max ( d_0, d_{\text{max}} ) $.
\end{proof}

\section{Concluding remarks}
Although we have only proved uniqueness of the Gibbs measure on the infinite regular tree for a sufficiently large degree $d$, our method could conceivably be extended to smaller values of $d$. With the aid of a computer we managed to check that for $q=3$ and $q=4$ and all $d\geq 2$ the map $F_d^{\circ 2}$ maps $P_c$ into $P_{\phi_d(-c)}$, where $\phi_d$ is the restriction of $-F_d^{\circ 2}$ to the line $\mathbb{R}\cdot \mathbf{1}$. 
It seems reasonable to expect that for other small values of $q$ a similar statement could be proved. A general approach is elusive so far.
It is moreover also not clear that $F_d(P_c)$ is convex, not even for $q=3$. 
In fact, for $q=3$ and $c$ large enough $F_3(P_c)$ is {\bf not} convex. 
But for reasonable values of $c$ it does appear to be convex. For larger values of $q$ this is even less clear.

Knowing that there is a unique Gibbs measure on the infinite regular tree is by itself not sufficient to design efficient algorithms to approximately compute the partition function/sample from the associated distribution on all bounded degree graphs.
One needs a stronger notion of decay of correlations, often called \emph{strong spatial mixing}~\cite{weitz06,Gamarnikcounting,GamarnikSSM,lu2013improved} or absence of complex zeros for the partition function near the real interval $[w,1]$~\cite{Barbook,PaR17,bencs,liu2019correlation}.
It is not clear whether our current approach is capable of proving such statements (these certainly do not follow automatically), but we hope that it may serve as a building block in determining the threshold(s) for strong spatial mixing and absence of complex zeros. 
We note that even for the case $w=0$, corresponding to proper colorings, the best known bounds for strong spatial mixing on the infinite tree~\cite{SSMcoloring} are still far from the uniqueness threshold. 
Very recently (after the current article was posted to the arXiv) these bounds have been significantly improved~\cite{chen2023strong}. 

\section*{Acknowledgement}
We are grateful to the anonymous referees for constructive and useful feedback.

\begin{appendix}
\begin{section}{Supplementary Mathematica code to Lemma~\ref{lemma:3variable}}
\label{sec: Mathematica code}
The functions $A_i$ for $i=1,2,3$ and $\Delta$ are defined as follows.

\begin{lstlisting}[language=Mathematica,caption={The functions $A_i$ and $\Delta$},label = {code: Ai Delta}]
A1[y1_, y2_, y3_, m_] := (m + 1) (1 - y1)/(1 + l y1 + y2 + (m - (l + 1)) y3)
A2[y1_, y2_, y3_, m_] := (m + 1) (1 - y2)/(1 + l y1 + y2 + (m - (l + 1)) y3)
A3[y1_, y2_, y3_, m_] := (m + 1) (1 - y3)/(1 + l y1 + y2 + (m - (l + 1)) y3)

Delta[y1_, y2_, y3_, m_] := (y1 y3 (m - l - 1) + (l + 1) y1 + (l + 1) y1 y2 - l y2) Exp[A1[y1, y2, y3, m]] 
+ (-y2 y3 (m - l - 1) - (l + 1) y1 y2 + y1 - 2 y2) Exp[A2[y1, y2, y3, m]] 
+ (y1 - y2) (1 - y3) (m - l - 1) Exp[A3[y1, y2, y3, m]]
\end{lstlisting}
The functions $y_i(t)$ are defined as follows.

\begin{lstlisting}[language=Mathematica,caption={The functions $y_i$}, label = {code: yi}]
{y1[t_], y2[t_], y3[t_]} = {y1, y2, y3} /. Solve[A1[y1, y2, y3, t] == C1 && A2[y1, y2, y3, t] == C2 && A3[y1, y2, y3, t] == C3, {y1, y2, y3}][[1]]
\end{lstlisting}

\begin{lstlisting}[language=Mathematica,caption={Verification that $y_i(q-1) = x_i$. This expression yields $\{x_1,x_2,x_3\}$}, label = {code: yi2}]
Simplify[{y1[q - 1], y2[q - 1], y3[q - 1]} /. {Rule[C1, A1[x1, x2, x3, q - 1]], Rule[C2, A2[x1, x2, x3, q - 1]], Rule[C3, A3[x1, x2, x3, q - 1]]}]
\end{lstlisting}

The function $r(t)$ can subsequently be found with the following code.

\begin{lstlisting}[language=Mathematica,caption={The function r}, label = {code: r}]
r[t_] = Simplify[Delta[y1[t], y2[t], y3[t], t] ((1 + t)/(1 + C2 - C3 + C1 l - C3 l + t + C3 t))^(-2)]
\end{lstlisting}

It can be observed that $r$ is indeed linear in $t$. To calculate $r(l+1)$ and the slope of $r$
we use the following piece of code.

\begin{lstlisting}[language=Mathematica,caption={The values of $r(l+1)$ and the slope of $r$}, label = {code: r slope}]
Simplify[{r[l + 1], Coefficient[r[t], t]}]
\end{lstlisting}

\end{section}
\end{appendix}

\bibliographystyle{alpha}
\bibliography{biblio}

\newcommand{\etalchar}[1]{$^{#1}$}
\begin{thebibliography}{CLMM23}

\bibitem[Bar16]{Barbook}
Alexander Barvinok.
\newblock {\em Combinatorics and complexity of partition functions}, volume~30
  of {\em Algorithms and Combinatorics}.
\newblock Springer, Cham, 2016.

\bibitem[Bax82]{BaxterZd}
R.~J. Baxter.
\newblock Critical antiferromagnetic square-lattice {P}otts model.
\newblock {\em Proc. Roy. Soc. London Ser. A}, 383(1784):43--54, 1982.

\bibitem[Bax86]{Baxterq}
R.~J. Baxter.
\newblock {$q$} colourings of the triangular lattice.
\newblock {\em J. Phys. A}, 19(14):2821--2839, 1986.

\bibitem[BBP21]{bencs2021limit}
Ferenc Bencs, Pjotr Buys, and Han Peters.
\newblock The limit of the zero locus of the independence polynomial for
  bounded degree graphs.
\newblock {\em arXiv preprint arXiv:2111.06451}, 2021.

\bibitem[BCKL13]{Borgsetal}
Christian Borgs, Jennifer Chayes, Jeff Kahn, and L\'{a}szl\'{o} Lov\'{a}sz.
\newblock Left and right convergence of graphs with bounded degree.
\newblock {\em Random Structures Algorithms}, 42(1):1--28, 2013.

\bibitem[BDPR21]{bencs}
Ferenc Bencs, Ewan Davies, Viresh Patel, and Guus Regts.
\newblock On zero-free regions for the anti-ferromagnetic {P}otts model on
  bounded-degree graphs.
\newblock {\em Ann. Inst. Henri Poincar\'{e} D}, 8(3):459--489, 2021.

\bibitem[BGG{\etalchar{+}}20]{blanca2020}
Antonio Blanca, Andreas Galanis, Leslie~Ann Goldberg, Daniel
  \v{S}tefankovi\v{c}, Eric Vigoda, and Kuan Yang.
\newblock Sampling in uniqueness from the {P}otts and random-cluster models on
  random regular graphs.
\newblock {\em SIAM Journal on Discrete Mathematics}, 34(1):742--793, 2020.

\bibitem[BR19]{Roz1}
Leonid~V. Bogachev and Utkir~A. Rozikov.
\newblock On the uniqueness of {G}ibbs measure in the {P}otts model on a
  {C}ayley tree with external field.
\newblock {\em J. Stat. Mech. Theory Exp.}, 7:073205, 76, 2019.

\bibitem[BV04]{boyd2004convex}
Stephen Boyd and Lieven Vandenberghe.
\newblock {\em Convex optimization}.
\newblock Cambridge University Press, Cambridge, 2004.

\bibitem[BW99]{brigt99}
G.R. Brightwell and P.~Winkler.
\newblock Graph homomorphisms and phase transitions.
\newblock {\em Journal of Combinatorial Theory, Series B}, 77(2):221--262,
  1999.

\bibitem[BW02]{bright02}
G.R. Brightwell and P.~Winkler.
\newblock Random colorings of a {C}ayley tree.
\newblock {\em Contemporary Combinatorics}, 10:247--276, 2002.

\bibitem[CLMM23]{chen2023strong}
Zongchen Chen, Kuikui Liu, Nitya Mani, and Ankur Moitra.
\newblock Strong spatial mixing for colorings on trees and its algorithmic
  applications.
\newblock {\em arXiv preprint arXiv:2304.01954}, 2023.

\bibitem[dBBR23]{deboer2020uniqueness}
David de~Boer, Pjotr Buys, and Guus Regts.
\newblock Uniqueness of the {G}ibbs measure for the 4-state anti-ferromagnetic
  {P}otts model on the regular tree.
\newblock {\em Combinatorics, Probability and Computing}, 32(1):158--182, 2023.

\bibitem[Eft22]{eft2020}
Charilaos Efthymiou.
\newblock On sampling symmetric {G}ibbs distributions on sparse random graphs
  and hypergraphs.
\newblock In {\em 49th {EATCS} {I}nternational {C}onference on {A}utomata,
  {L}anguages, and {P}rogramming}, volume 229 of {\em LIPIcs. Leibniz Int.
  Proc. Inform.}, pages Art. No. 57, 16. Schloss Dagstuhl. Leibniz-Zent.
  Inform., Wadern, 2022.

\bibitem[EGH{\etalchar{+}}19]{SSMcoloring}
Charilaos Efthymiou, Andreas Galanis, Thomas~P. Hayes, Daniel
  \v{S}tefankovi\v{c}, and Eric Vigoda.
\newblock Improved strong spatial mixing for colorings on trees.
\newblock In {\em Approximation, randomization, and combinatorial optimization.
  {A}lgorithms and techniques}, volume 145 of {\em LIPIcs. Leibniz Int. Proc.
  Inform.}, pages Art. No. 48, 16. Schloss Dagstuhl. Leibniz-Zent. Inform.,
  Wadern, 2019.

\bibitem[FV17]{friedli2017}
Sacha Friedli and Yvan Velenik.
\newblock {\em Statistical Mechanics of Lattice Systems: a Concrete
  Mathematical Introduction}.
\newblock Cambridge University Press, 2017.

\bibitem[Geo88]{Georgii}
Hans-Otto Georgii.
\newblock {\em Gibbs measures and phase transitions}, volume~9 of {\em De
  Gruyter Studies in Mathematics}.
\newblock Walter de Gruyter \& Co., Berlin, 1988.

\bibitem[GGY18]{PottsLeslie3}
Andreas Galanis, Leslie~Ann Goldberg, and Kuan Yang.
\newblock Uniqueness for the 3-state antiferromagnetic {P}otts model on the
  tree.
\newblock {\em Electron. J. Probab.}, 23:Paper No. 82, 43, 2018.

\bibitem[GK12]{Gamarnikcounting}
David Gamarnik and Dmitriy Katz.
\newblock Correlation decay and deterministic {FPTAS} for counting colorings of
  a graph.
\newblock {\em J. Discrete Algorithms}, 12:29--47, 2012.

\bibitem[GKM15]{GamarnikSSM}
David Gamarnik, Dmitriy Katz, and Sidhant Misra.
\newblock Strong spatial mixing of list coloring of graphs.
\newblock {\em Random Structures Algorithms}, 46(4):599--613, 2015.

\bibitem[GRR17]{Roz3}
D.~Gandolfo, M.~M. Rahmatullaev, and U.~A. Rozikov.
\newblock Boundary conditions for translation-invariant {G}ibbs measures of the
  {P}otts model on {C}ayley trees.
\newblock {\em J. Stat. Phys.}, 167(5):1164--1179, 2017.

\bibitem[G{\v{S}}V15]{gal15}
Andreas Galanis, Daniel {\v{S}}tefankovi{\v{c}}, and Eric Vigoda.
\newblock Inapproximability for antiferromagnetic spin systems in the tree
  nonuniqueness region.
\newblock {\em Journal of the ACM (JACM)}, 62(6):1--60, 2015.

\bibitem[G{\v{S}}V16]{gal16}
Andreas Galanis, Daniel {\v{S}}tefankovi{\v{c}}, and Eric Vigoda.
\newblock Inapproximability of the partition function for the antiferromagnetic
  {I}sing and hard-core models.
\newblock {\em Combinatorics, Probability and Computing}, 25(4):500--559, 2016.

\bibitem[Jon02]{jon02}
Johan Jonasson.
\newblock Uniqueness of uniform random colorings of regular trees.
\newblock {\em Statistics \& Probability Letters}, 57(3):243--248, 2002.

\bibitem[KR17]{Roz2}
Christof K\"{u}lske and Utkir~A. Rozikov.
\newblock Fuzzy transformations and extremality of {G}ibbs measures for the
  {P}otts model on a {C}ayley tree.
\newblock {\em Random Structures Algorithms}, 50(4):636--678, 2017.

\bibitem[KRK14]{Roz4}
C.~K\"{u}lske, U.~A. Rozikov, and R.~M. Khakimov.
\newblock Description of the translation-invariant splitting {G}ibbs measures
  for the {P}otts model on a {C}ayley tree.
\newblock {\em J. Stat. Phys.}, 156(1):189--200, 2014.

\bibitem[LLY13]{licor}
Liang Li, Pinyan Lu, and Yitong Yin.
\newblock Correlation decay up to uniqueness in spin systems.
\newblock In {\em Proceedings of the twenty-fourth annual ACM-SIAM symposium on
  Discrete algorithms}, pages 67--84. SIAM, 2013.

\bibitem[LSS0]{liu2019correlation}
Jingcheng Liu, Alistair Sinclair, and Piyush Srivastava.
\newblock Correlation decay and partition function zeros: Algorithms and phase
  transitions.
\newblock {\em SIAM Journal on Computing}, 0(0):FOCS19--200--FOCS19--252, 0.

\bibitem[LY13]{lu2013improved}
Pinyan Lu and Yitong Yin.
\newblock Improved {FPTAS} for multi-spin systems.
\newblock In {\em Approximation, Randomization, and Combinatorial Optimization.
  Algorithms and Techniques}, pages 639--654. Springer, 2013.

\bibitem[PdLM83]{Peruggietalgeneral}
Fulvio Peruggi, Francesco di~Liberto, and Gabriella Monroy.
\newblock The {P}otts model on {B}ethe lattices. {I}. {G}eneral results.
\newblock {\em J. Phys. A}, 16(4):811--827, 1983.

\bibitem[PdLM87]{Peruggietaldiagrams}
Fulvio Peruggi, Francesco di~Liberto, and Gabriella Monroy.
\newblock Phase diagrams of the {$q$}-state {P}otts model on {B}ethe lattices.
\newblock {\em Phys. A}, 141(1):151--186, 1987.

\bibitem[Pot52]{potts}
R.~B. Potts.
\newblock Some generalized order-disorder transformations.
\newblock {\em Proc. Cambridge Philos. Soc.}, 48:106--109, 1952.

\bibitem[PR17]{PaR17}
Viresh Patel and Guus Regts.
\newblock Deterministic polynomial-time approximation algorithms for partition
  functions and graph polynomials.
\newblock {\em SIAM J. Comput.}, 46(6):1893--1919, 2017.

\bibitem[Roz21]{Roz21}
U.~A. Rozikov.
\newblock Gibbs measures of {P}otts model on {C}ayley trees: a survey and
  applications.
\newblock {\em Rev. Math. Phys.}, 33(10):Paper No. 2130007, 58, 2021.

\bibitem[Sok05]{Sokaltutte}
Alan~D. Sokal.
\newblock The multivariate {T}utte polynomial (alias {P}otts model) for graphs
  and matroids.
\newblock In {\em Surveys in combinatorics 2005}, volume 327 of {\em London
  Math. Soc. Lecture Note Ser.}, pages 173--226. Cambridge Univ. Press,
  Cambridge, 2005.

\bibitem[SS97]{SalasSokal}
Jes\'{u}s Salas and Alan~D. Sokal.
\newblock Absence of phase transition for antiferromagnetic {P}otts models via
  the {D}obrushin uniqueness theorem.
\newblock {\em J. Statist. Phys.}, 86(3-4):551--579, 1997.

\bibitem[SS14]{Slysun}
Allan Sly and Nike Sun.
\newblock Counting in two-spin models on {$d$}-regular graphs.
\newblock {\em Ann. Probab.}, 42(6):2383--2416, 2014.

\bibitem[SST14]{sst14}
Alistair Sinclair, Piyush Srivastava, and Marc Thurley.
\newblock Approximation algorithms for two-state anti-ferromagnetic spin
  systems on bounded degree graphs.
\newblock {\em Journal of Statistical Physics}, 155(4):666--686, 2014.

\bibitem[Wei06]{weitz06}
Dror Weitz.
\newblock Counting independent sets up to the tree threshold.
\newblock In {\em Proceedings of the thirty-eighth annual ACM symposium on
  Theory of Computing}, pages 140--149, 2006.

\end{thebibliography}

\end{document}